\documentclass[10pt]{amsart}
\usepackage{amsmath}
\usepackage{amssymb}
\usepackage{epsfig}
\usepackage{graphics,graphicx}
\usepackage{lineno}
\usepackage{color}
\usepackage[unicode=true,pdfusetitle,
   bookmarks=true,
   bookmarksnumbered=false,
   bookmarksopen=false,
   breaklinks=false,
   pdfborder={0 0 1},
   backref=false,
   colorlinks=false]{hyperref}

\hypersetup{pdfauthor={Name}}
\PassOptionsToPackage{unicode}{hyperref}
\PassOptionsToPackage{naturalnames}{hyperref}
\usepackage{hyperref,varioref}
\usepackage{mathrsfs}
\usepackage{amsfonts}
\usepackage[utf8]{inputenc}
\usepackage{caption}
\usepackage{verbatim}
\usepackage{amssymb}
\usepackage{empheq}
\usepackage{latexsym}
\usepackage{amsmath, amsthm, amssymb, cases}
\usepackage{color}
\usepackage{xcolor}

\def\la{\lambda}

\newtheorem{corollary}{Corollary}[section]

\newtheorem{lemma}[corollary]{Lemma}
\newtheorem{proposition}[corollary]{Proposition}
\newtheorem{remark}[corollary]{Remark}
\newtheorem{theorem}[corollary]{Theorem}
\newcommand{\mylabel}[1]{\label{#1}
            \ifx\undefined\stillediting
            \else \fbox{$#1$}\fi }
\newcommand{\BE}{\begin{equation}}

\newcommand{\EEQ}{\end{equation}}
\newcommand{\rfb}[1]{\mbox{\rm
   (\ref{#1})}\ifx\undefined\stillediting\else:\fbox{$#1$}\fi}

\newfont{\Blackboard}{msbm10 scaled 1200}

\newfont{\roma}{cmr10 scaled 1200}

\def\CC{\rm \hbox{C\kern-.56em\raise.4ex
         \hbox{$\scriptscriptstyle |$}\kern+0.5 em }}

\def\r{\rho}
\def\s{\sigma}

\def\o{\omega}
\def\p{\phi}

\def\O{\Omega}

\newcommand{\mm}    {{\hbox{\hskip 0.5pt}}}

\newcommand{\bluff} {{\hbox{\raise 15pt \hbox{\mm}}}}
\makeatletter
\def\section{\@startsection {section}{1}{\z@}{-3.5ex plus -1ex minus
    -.2ex}{2.3ex plus .2ex}{\large\bf}}
\makeatother
\def\be{\begin{equation}}
\def\ee{\end{equation}}

\def\ds{\displaystyle}
\renewcommand\({\left(}
\renewcommand\){\right)}
\renewcommand\[{\left[}
\renewcommand\]{\right]}

\def \R {{\mathbb{R}}}
\def \N {{\mathbb{N}}}
\def \C {{\mathbb{C}}}
\begin{document}

\thispagestyle{empty}
\title[Controllability of a heat system]{Null boundary controllability of a one-dimensional heat equation with internal point masses and variable coefficients}
\date\today
\author{Ka\"{i}s Ammari}
\address{UR Analysis and Control of PDEs, UR 13ES64, Department of Mathematics, Faculty of Sciences of Monastir, University of Monastir, Tunisia}
\email{kais.ammari@fsm.rnu.tn}

\author{Hedi Bouzidi}
\address{UR Analysis and Control of PDEs, UR 13ES64, Department of Mathematics, Faculty of Sciences of Monastir, University of Monastir, Tunisia}
\email{hedi.bouzidi@fst.utm.tn}

\begin{abstract}
In this paper, we consider a linear hybrid system which is composed of $N+1$ non-homogeneous thin rods connected by
$N$ interior-point masses with a Dirichlet boundary condition on the
left end, and Dirichlet control on the right end. Using a detailed spectral analysis
and the moment theory, we prove that this system
is null controllable at any positive time $T$. To this end, firstly, we implement the Wronskian technique to obtain the characteristic equation for the eigenvalues $(\lambda_{n})_{n\in\N^*}$ associated
with this system. Secondly, we provide that the eigenvalues $(\lambda_{n})_{n\in\N^*}$ interlace those of the $N+1$ decoupled rods with homogeneous Dirichlet boundary conditions, and satisfy the so-called Weyl's asymptotic formula. Finally, we establish sharp asymptotic estimates of the eigenvalues $(\lambda_{n})_{n\in\N^*}$. As consequence, on one hand, we prove a uniform lower bound for the spectral gap. On
another hand, we  derive the equivalence between the $\mathcal{H}$-norm of the eigenfunctions and their first derivative at the right
end. As an application of our spectral analysis, we also present new controllability result for the Schr\"{o}dinger equation with an internal point mass and Dirichlet control on the left end.
\end{abstract}

\subjclass[2010]{35K05, 93B05, 93B55, 93B60}
\keywords{Boundary control, thin rods, nonhomogeneous, point masses, Wronskians, Weyl's formula, moments.}

\maketitle

\tableofcontents


\section{Introduction and main results} \label{secintro}
\setcounter{equation}{0}
The boundary controllability of the so-called "hybrid systems"  has been extensively investigated for several decades. This was pioneered in \cite{HansenZ} by the moment method for a hybrid system composed of two vibrating strings connected by a point mass. Since then, the controllability of hybrid models for systems of Rayleigh and Euler-Bernoulli beams with interior point masses was considered in \cite{Castro97B, Castro98, Mercier2, Mercier3}, see also \cite{Hedi2019, Castro2000, Mercier1} and references therein. More recently, a variety of other hybrid models for thin rods, quantum boxes and other elastic systems involving point masses have been studied along similar lines.  In particular, see \cite{Hedi2018, HansenM1} for a heat equation with internal point masses, \cite{Ammari2015, Avdonin2020, Hansen} for a Schr\"odinger equation with internal point masses, and \cite{Ammari2016, Ammari2017, Ammari2010, Ammari2015bis, Ammari2012, Ammari2001, Ammari2000, Avdonin2015, Avdonin2016, Amara2019, Castro97} for networks of strings with attached masses.

\medskip

In this paper, we study the boundary null controllability of the
temperature of a linear hybrid system consisting of $N+1$
non-homogeneous rods connected by $N$ point masses. Assume the $N+1$
non-homogeneous rods occupy the interval $\O=(0, L),~L>0,$ of
the $x$-axis and they are connected by $N$ masses $M_j> 0$ at the
points $\ell_j, j = 1, . . .,N,$ where $0 = \ell_0 < \ell_1 <... < \ell_N <
\ell_{N+1}=L$. We partition the domain $\O$ as follows:
$$\O:=\bigcup_{j=0}^{N}\Big\{\O_j\cup\{\ell_{j+1}\}\cup\O_{j+1}\Big\},
~\O_j=(\ell_j,\ell_{j+1}), ~j = 0, . . .,N.$$ By means of the scalar
functions
\begin{eqnarray*} &&u_j(t,x),~t>0,~x\in \O_j, j = 0, . . .,N,\\
&&z_j(t),~t>0,~ j = 1, . . .,N, \end{eqnarray*} we describe the temperature of
the rods $\O_j,$ and the temperature of the points
masses $\ell_j$, respectively. The linear equation
modeling heat flow of such a system is as follows:
\begin{equation}\label{eqs:1}\left\{\begin{array}{ll}
\(\rho_{j}(x)\partial_t u_j-\partial_x\(\sigma_{j}(x)\partial_x u_j\)+q_{j}(x)u_j\)(t,x)=0,& t>0,~ x\in {\Omega}_j,~j = 0,...,N, \\
u_{j-1}(t, \ell_j)=z_j(t)=u_{j}(t, \ell_j),&t>0,~j=1,...,N, \\
\(\sigma_{j}(\ell_j)\partial_x u_j-\sigma_{j-1}(\ell_j)\partial_x u_{j-1}\)(t,\ell_j)= M_j\partial_t z_{j}(t),&t>0,~j=1,...,N,  \\
u_{0}(t,\ell_{0})=u_{0}(t,0)=0, & t>0, \end{array}\right.\end{equation} with the control \begin{equation}
u_{N}(t,\ell_{N+1})=u_{N}(t,L)=h(t),~~t>0\label{eqs:2}\end{equation} and the
initial conditions at $t=0$ \begin{eqnarray}\label{eqs:3} \left\{
\begin{array}{ll}
  u_j(0,x)=u^{0}_j,~x\in\O_j,&j=0,...,N, \\
  z_j(0)=z^{0}_j,&j = 1, ...,N.
\end{array}
 \right.
\end{eqnarray} In System \eqref{eqs:1}, for each $j=0,...,N,$
the coefficients $\r_{j}(x)$ and $\s_{j}(x),$
represent respectively the density and thermal conductivity of the
rods. The potentials are denoted by the functions $q_{j}(x)$,
$j=0,...,N$. Throughout this paper, we assume that the coefficients
\begin{equation}\label{eqs:4}
\r_{j},~\s_{j} \in H^{2}(\O_j),~q_{j}\in H^{1}(\O_j),
~~j=0,...,N,
\end{equation}
and there exist constants $\r,~\s>0$, such that \begin{equation}
\r_j(x)\geq\r,~~\s_j(x)\geq\s,~~ q_j(x)\geq0,~~x
\in\overline{\O}_j,~j=0,...,N. \label{eqs:5}\end{equation}
To state our main null controllability result for system
\eqref{eqs:1}-\eqref{eqs:3}, we need some definitions and notations.
We denote by $\underline{u}:=\(u_j\)_{j=0}^N=(u_0,...,u_N)$ the
functions on $\O$ taking their values in $\C$ and let $u_j(x)$ be
the restriction of $\underline{u}$ to $\O_j$, $j=0,...,N$.
Let us define the following Hilbert space \begin{equation}
\mathcal H =\prod_{j=0}^{N} L^{2}_{\r_j}(\O_j)\times\R^{N},
\nonumber\end{equation}
which is endowed with the Hilbert structure \begin{eqnarray}
\Big\langle\left(\underline{u},\dot{\underline h} \right)^\top ,\left(
                                    \underline{v}, \dot{\underline k} \right)^\top\Big\rangle_{\mathcal{H}}:=\sum_{j=0}^{N}\int_{\O_j}u_jv_j\r_j(x)dx+\sum_{j=1}^{N}M_{j}h_jk_j,\label{eqs:6}
\end{eqnarray} where $\dot{\underline h}=\(h_j\)_{j=1}^N, ~\dot{\underline k}=\(k_j\)_{j=1}^N\in\R^{N },$ and $^\top$ denotes transposition. Hereafter, we use the notation $f(x)g(x):=fg(x).$ Our first main result is stated as
follows:
\begin{theorem}\label{MAIN}
Assume that the coefficients $\r_{j}(x)$, $\s_{j}(x)$ and
$q_{j}(x)$ satisfy \eqref{eqs:4} and \eqref{eqs:5}. Let $
T > 0$, then for any initial data $U^0:=\(\(u^{0}_j\)_{j=0}^N,\(z^{0}_j\)_{j=1}^{N}
\)^{\top}\in
\mathcal{H}$ there exists a control $h \in H^{1}(0, T),$ given explicitly by the expression $(\ref{eqref4s12})$, such
that the solution $U:=\(\(u_j\)_{j=0}^N,\(z_j\)_{j=1}^{N}\)^{\top}$ of
the control system \eqref{eqs:1}-\eqref{eqs:3} satisfies
 \begin{eqnarray*} \left\{
\begin{array}{ll}
  u_j(T,x)=0,&x\in\O_j,~
j=0,...,N, \\
  z_j(T)=0,& j=1,...,N.
\end{array}
 \right.
\end{eqnarray*}
\end{theorem}
Our approach is mainly based on a precise analysis of the eigenvalue and eigenfunction asymptotics of the corresponding second order eigenvalue problem, and the general moment theory \cite{RussellF,RussellF1}. Firstly, we implement the Wronskian technique (e.g., \cite[Chapter
1]{Freiling} and \cite[Chapter 1]{TITCHMARSH}), to obtain the characteristic equation
for the eigenvalues $(\la_{n})_{n\in\N^*}$ associated with
System \eqref{eqs:1}-\eqref{eqs:2} (see Theorem \ref{wronskian}). Secondly, we provide the following interlacing property for the eigenvalues $(\la_{n})_{n\in\N^*}$ (see Theorem \ref{interlcing}), \begin{equation} 0<\la_{1}\leq\mu^{N,D}_{1}~~\hbox{ and }~~\mu_{n}^{N,D}
\leq\la_{n+1}\leq\mu^{N,D}_{n+1},~\forall n\in\N^*,
 \label{eqs:11najem}
\end{equation} where
$\{\mu_n^{N,D}\}_{1}^{\infty}=\bigcup_{j=0}^{N}
\{\widehat\mu_n^{j,D}\}_{1}^{\infty}$ are the eigenvalues of the $N+1$
 decoupled
rods with homogeneous Dirichlet boundary conditions. Then, we establish the Weyl's type asymptotic
formula for the eigenvalues $(\la_{n})_{n\in\N^*}$ : \begin{equation}
\label{Eqs:7} \lim_{n\to\infty}\frac{\la_{n}}{n^2\pi^2}=\gamma=
{\(\sum_{j=0}^{N}\int_{\ell_j}^{\ell_{j+1}}\sqrt{\frac{\r_j(x)}{\s_j(x)}}dx\)^{-2}}.\end{equation}
Finally, using the interlacing property \eqref{eqs:11najem} and the Weyl's formula \eqref{Eqs:7}, we obtain sharp asymptotic estimates of the eigenvalues $(\la_{n})_{n\in\mathbb{N}^{*}}$. Roughly speaking, we show that the set of eigenvalues $\{\la_n\}_{n\in\N^*}$ is asymptotically splits into $N+1$ branches $\{\lambda_{n}^j\}_{n\in\N^*}$, $j=0,...,N$, in the sense that
\be \left|\sqrt{\lambda_{n+1}^j}-\sqrt{\widehat \mu_n^{j,D}}\right|\asymp \dfrac{1}{n} ~\hbox{ for } ~\widehat \mu_n^{j,D}\not=\lambda_{n+1}^j, ~j=0,...,N. \label{Eqs:7habib} \ee
As consequence, on one hand, we prove a uniform lower bound for the spectral gap (see Theorem \ref{Weyl}), namely, \be
{\la_{n+1}}-{\la_{n}}\geq2\gamma\min_{j=0,...,N-1}
\left\{\dfrac{\left({\r_{j}\s_{j}(\ell_{j+1})}\right)^{-\frac{1}{2}}}{M_{j+1}{\o_j^*}^2},
\dfrac{\left(\r_{N}\s_{N}(\ell_N)\right)^{-\frac{1}{2}}}{M_{N}{\o_N^*}^2}\right\}, \hbox{ as }n\to\infty,
\label{eqsssV4kais:31dagerz}\ee
where $\o_j^*=\int_{\Omega_j}
\sqrt{\frac{\r_{j}(t)}{\s_{j}(t)}}dt,~j=0,...,N$.  On another hand, we  derive the equivalence between the $\mathcal{H}$-norm of the eigenfunctions
$\({\Phi}_n\)_{n\in\N^*}$ and their first derivative at the right
end $x=L$ (see Proposition \ref{Lem1asymptotics210v4equi}), that is, \be  \dfrac{\left\|{\Phi}_n\right\|_{\mathcal{H}}}{\left|\s_{N}(L){\Phi}_n'(L)\right|}
\sim\sqrt{\dfrac{\o_N^*}{{2}}}\dfrac{\gamma\left(\r_{N}\s_{N}(\ell_N)\right)^{-\frac{1}{4}}}
{{n\pi}},\hbox{ as }n\to\infty.\label{eqsssV4kais:31dager}\ee
Using these results, we reduce the control
problem \eqref{eqs:1}-\eqref{eqs:2} into an equivalent moment problem which will be solved by the
general moment theory developed by Fattorini and Russell \cite{RussellF,RussellF1}. As an application of our spectral analysis, we also present new controllability result for the following Schr\"{o}dinger equation with an internal point mass:\be \label{Scontrolanis} \left\{
\begin{array}{llll}
i\partial_t u_j(t,x)-\partial_{xx} u_j(t,x)=0,& t>0,~ x\in \(\ell_j,\ell_{j+1}\),~j = 0,1, \\
u_{0}(t, \ell_1)=z(t)=u_{1}(t, \ell_1),&t>0, \\
\(\partial_x u_1-\partial_x u_{0}\)(t,\ell_1)= i\partial_t z(t),&t>0,  \\
u_{0}(t,\ell_{0})=u_{0}(t,0)=h(t), ~~u_{1}(t,\ell_2)=0& t>0,\\
u_{0}^0=u_{0}(0,x),~u_{1}^0=u_{1}(0,x),~z^0=z(0),\end{array}
 \right.\ee
where $0 = \ell_0 < \ell_1 < \ell_2=1$, ${i}^2=-1$ is the imaginary unit, $h(t)$ is the control. To this end, we assume that
\be \ell_1\not\in\left\{\frac{p}{p+1}~:~p\in\N^*\right\}.\label{brous}
\ee
We then prove that the exact controllability of \eqref{Scontrolanis} can not hold in an asymmetric control space. Namely, we enunciate the
following result:
\begin{theorem}
\label{hjAnis} Let $T>0,$ and assume that \eqref{brous} holds. Then, for every
 $\(\underline{u}^0 :=(u_0^0,u_1^0),z^0\)^{\top}\in H^{-1}(0,1) \times \mathbb{C}$, there
exists a control $h(t)\in L^{2}(0,T)$ such that the solution \\$U:=\big(u_0(t,x),u_1(t,x),z(t)\big)^{\top}$ of the
problem \eqref{Scontrolanis}
 satisfies
\begin{equation*}
u_0(T,x)=u_1(T,x)=z(T)=0.
\end{equation*}
\end{theorem}

\medskip

Let us now describe the existing results on null boundary controllability of System \eqref{eqs:1}-\eqref{eqs:3}. When $M_j=0$ for all $j\in\{0,...,N\}$, we recover the continuity condition of $u_{j}(t,x)$ at the points $\ell_j$, $j=1,...,N,$ and the classical heat equation with variable
coefficients occupying the interval $\O$ without point masses.
In this context,  the null controllability of Problem
\eqref{eqs:1}-\eqref{eqs:3} (for $M_j=0,$ $j=0,...,N$) has been widely studied since the pioneering works of Fattorini and Russell \cite{RussellF,RussellF1}.  We refer to \cite{CoronN, Cannarsa, CaraZuazua} for related results on null controllability  of the heat equation with variable
coefficients,  also \cite{Martin} and references therein.  In the case of a single attached mass, the wellposedness of System \eqref{eqs:1}-\eqref{eqs:3} with general constant coefficients was firstly studied by Hansen and Martinez \cite{HansenM}. Later on, null controllability result at time $T > 0 $ for System
\eqref{eqs:1}-\eqref{eqs:3} with $N=1$,
$\r_{j}(x)\equiv\s_{j}(x)\equiv1$ and $q_{j}(x)\equiv0~(j=1,2)$ was proved by the same authors in \cite{HansenM1}.  The previous controllability result has been
extended  in \cite{Hedi2018} to the case of System \eqref{eqs:1}-\eqref{eqs:3} with $N=1$.  The method of
the both papers is based on a precise analysis of the eigenvalue and eigenfunction asymptotics, and the general moment theory \cite{RussellF,RussellF1}. In the case of several attached masses, under strong assumptions on the regularity of the coefficients $\r_{j},~\s_{j}$  and $ q_j$, null controllability result at any positive time $T$ for System  \eqref{eqs:1}-\eqref{eqs:3} has been recently established by Avdonin et al. \cite{Avdonin2021}. Their proof uses the transmutation method (e.g; see \cite{CaraZuazua}), which relates the null controllability of System \eqref{eqs:1}-\eqref{eqs:3} to the exact controllability of the vibrating string with attached masses \cite{Avdonin2015}.

\medskip

It is worth mentioning that the sharp asymptotic estimate \eqref{Eqs:7habib} was proved in the seminal paper by Hansen and Zuazua \cite{HansenZ} in the case of the string equation with interior point masses and constant physical parameters. Let us also underline the reference \cite{HansenM1}, where the authors prove the uniform lower bound \eqref{eqsssV4kais:31dagerz} of the spectral gap associated with System
\eqref{eqs:1}-\eqref{eqs:3} in the case $N=1$,
$\r_{j}(x)\equiv\s_{j}(x)\equiv1$ and $q_{j}(x)\equiv0~(j=1,2)$. Recently, Avdonin and
Edwards \cite{Avdonin2015} were able to give partial answer concerns the asymptotic behavior of the eigenvalues $(\la_n)_{n\in\N^*}$ of the spectral problem associated with \eqref{eqs:1}-\eqref{eqs:2} for $\r_{j}(x)\equiv\s_{j}(x)\equiv1$. They established that the eigenvalues $(\la_n)_{n\in\N^*}$  satisfy the asymptotes $\left|\sqrt{\la_n}-\frac{n\pi}{\ell_j}\right|=O(\frac{1}{n})$, $j=0,...,N$. This result has been proved as a consequence of the mini-max argument by applying the Rouch\'{e}'s theorem to the spectral problem associated to \eqref{eqs:1}-\eqref{eqs:2} for $\r_{j}(x)\equiv\s_{j}(x)\equiv1$ and $q_{j}(x)\equiv0$. The exact boundary controllability of the Schr\"{o}dinger model \eqref{Scontrolanis} with $\ell_1=\dfrac{1}{2},$ was studied by Hansen in \cite{Hansen}. In that paper, the author consider System \eqref{Scontrolanis} with a Dirichlet boundary condition on the
left end $\ell_0=0$, and either Dirichlet or Neumann boundary control on the right end $\ell_2=1$. In the case of Dirichlet
control,  the author proves that the exact controllability space is $H^{-1}(0,1)\times \mathbb{C}$. While, in the case of Neumann control, the exact controllability space is asymmetric with respect to the
point mass in the sense that the regularity is one degree higher
on the side of the point mass opposite the control.
Later on, Avdonin and
Edwards \cite{Avdonin2020} studied the Dirichlet boundary controllability of the Schr{\"{o}}dinger equation with internal point masses and various homogeneous boundary conditions at one end. Somewhat surprisingly, one of their main results is that System \eqref{Scontrolanis} is exactly  controllable in $H^{-1}\(0,1\)\times \mathbb{C}$ if and only if
Condition \eqref{brous} is not satisfied. Their proof uses a diophantine approximation argument. As consequence, if \eqref{brous} is fulfilled, the exact controllability space is asymmetric in the sense that the regularity is $H^{-1}\(0,\ell_1\)$ on the left side of the point mass
and $H^{-2}\(\ell_1,1\)$ to the right of the point mass. As we will see in subsection \ref{sec4-2},  the exact controllability space does not depend on the diophantine approximation
of $\ell_1$. In forthcoming paper \cite{Ammari20021}, we consider  the exact controllability of the Schr{\"{o}}dinger equation with internal point masses and variable coefficients. In that paper, we assume a Dirichlet boundary condition at one end, and Neumann boundary control on the other end. We prove that this system is  exact controllable in asymmetric spaces whose the regularity to the right of each mass exceeds the regularity to the left by one Sobolev order from the controlled end.

This paper is organized as follows:
In Section \ref{Sec2}, we establish some results which will be
used along this work. In subsection \ref{sec1-1}, we show that the eigenvalues $\(\la_n\)_{n\in\N^*}$ associated with System \eqref{eqs:1}-\eqref{eqs:3} are simple, and we characterize the corresponding eigenfunctions. In subsection \ref{Sec4-1b}, we investigate the well-posedness of the heat model \eqref{eqs:1}-\eqref{eqs:3}. In Section \ref{Sect3}, we investigate the main properties of all the eigenvalues $(\la_{n})_{n\in\mathbb{N}^{*}}$: First,
in Subsection \ref{Sec3-1}, we establish the characteristic
equation for the eigenvalues  $(\la_{n})_{n\in\mathbb{N}^{*}}$. Subsection \ref{Sec3-2}, is devoted to the interlacing property  \eqref{eqs:11najem}, and the Weyl's formula \eqref{Eqs:7}. In subsection \ref{Sec3-3}, we obtain sharp asymptotic estimates of the eigenvalues $(\la_{n})_{n\in\mathbb{N}^{*}}$. The gap condition \eqref{eqsssV4kais:31dagerz}, and the equivalence \eqref{eqsssV4kais:31dager} are concluded
as a consequence. Finally,
in Section \ref{Sec4}, we prove our main results, namely the null controllability of System
\eqref{eqs:1}-\eqref{eqs:3}, and then, the exact controllability of the Schr\"{o}dinger model \eqref{Scontrolanis}.

\section{Characterization of the eigenelements and Well-posedness} \label{Sec2}
\setcounter{equation}{0}
\subsection{Characterization of the eigenelements}\label{sec1-1} In this subsection, we establish some spectral results which will be
used along this work. First, we prove the existence and uniqueness
of solutions for the initial value problems associated with spectral problem:
 \begin{eqnarray}\(\mathcal{P}_N\) \label{eqs:8}\left\{\begin{array}{ll}
 -(\s_{j}(x)\p_j')'+q_{j}(x)\p_j=\la \r_{j}(x)\p_j, ~~x\in \O_j,&j=0,...,N, \\
\p_{j-1}(\ell_{j})=\p_{j}(\ell_{j}),&j=1,...,N, \\
\s_{j-1}\p_{j-1}'(\ell_{j})-\s_{j}\p_{j}'(\ell_{j})=M_{j}\la \p_{j-1}(\ell_{j}),&j=1,...,N,\\
\p_1(\ell_1)=\p_1(0)=0,~ \p_{N}(\ell_{N+1})=\p_N(L)=0.
 \end{array}\right.\end{eqnarray} Then, we study the asymptotic properties of these solutions.  As consequence, we show that the eigenvalues $\(\la_n\)_{n\in\N^*}$ of Problem $\(\mathcal{P}_N\)$
\eqref{eqs:8} are simple, and we characterize the associated eigenfunctions.
To this end, let us introduce the following Hilbert space \begin{equation*}
\mathcal{V}=\Bigg\{\underline{u}:=\(u_j\)_{j=0}^N\in\prod_{j=0}^{N} H^{1}(\O_j):\begin{cases} u_0(0)=u_N(L)=0,\\
u_{j-1}(\ell_{j})=u_{j}(\ell_{j}),~ j=1,...,N,
\end{cases}\Bigg\}\Bigg\},
\end{equation*} which endowed with the Hilbert structure
\begin{equation*}\langle\underline{u},\underline{v}\rangle_{\mathcal V}=\sum_{j=0}^{N}\int_{\O_j}u_j'v_j'\r_j(x)dx,~~ \underline{v}=\(v_j\)_{j=0}^N.\end{equation*} We consider the following closed subspace of
$\mathcal{V} \times \mathbb{R}^{N}$, \begin{equation*}
\mathcal{W}=\Bigg\{(\underline{u},\dot{\underline{z}})^\top\in
\mathcal{V}
\times \mathbb{R}^{N}: \begin{cases} \dot{\underline z}:=\(z_j\)_{j=1}^N=(z_1,...,z_N),\\
z_{j-1}=u_{j-1}(\ell_{j})=u_{j}(\ell_{j}),~j=1,...,N
\end{cases}\Bigg\}\Bigg\}, \end{equation*}
which is densely and continuously embedded in the space
$\mathcal{H}$. In the sequel we introduce the operator $\mathcal{A}$
defined in $\mathcal{H}$ by setting
\begin{equation}\label{eqss:1}
\mathcal{A}u =\Bigg(
\Big(\dfrac{1}{\r_j(x)}\(-(\s_{j}(x)u_j')'+q_{j}(x)u_j\)\Big)_{j=0}^N,
\Big(\dfrac{1}{M_{j}}\(\s_{j-1}u_{j-1}'(\ell_{j})-\s_{j}u_{j}'(\ell_{j})\)\Big)_{j=1}^N
\Bigg)^\top,
\end{equation}
where $u=(\underline{u},\dot{\underline{z}})^\top$ on the domain $$
{\mathcal D}(\mathcal{A})=\big\{(\underline{u},\dot{\underline{z}})^\top\in
\mathcal{W} :\underline{u}=\(u_j\)_{j=0}^N,~ u_j\in
H^{2}(\O_j),~j=1,...,N\big\},
$$
which is dense in $\mathcal{H}$. Obviously, the spectral problem
$\(\mathcal{P}_N\)$ \eqref{eqs:8} is equivalent to the following
problem
$$\mathcal{A}\Phi=\la\Phi,~~\Phi:=
\(\big({\p}_j(x,\la)\big)_{j=0}^N,~
\big({\p}_j(\ell_j,\la)\big)_{j=1}^N\)^\top\in
{\mathcal D}(\mathcal{A}),$$ i.e., the eigenvalues $\(\la_n\)_{n\in\N^*}$, of the
operator $\mathcal{A}$ and Problem $\(\mathcal{P}_N\)$ \eqref{eqs:8}
coincide together with their multiplicities. Moreover, there is a
one-to-one correspondence between the eigenfunctions, \begin{equation}
\Phi_{n}:=\(\big({\p}_j(x,\la_n)\big)_{j=0}^N,~
\big({\p}_j(\ell_j,\la_n)\big)_{j=1}^N\)^\top\xrightleftharpoons{\hspace{1cm}}
\underline{\p}_{n}:=\big({\p}_j(x,\la_n)\big)_{j=0}^N,~n\in\N^*.\label{eqss:2}\end{equation}
\begin{lemma}
\label{operator} The linear operator $\mathcal{A}$ is positive and
self-adjoint such that $\mathcal{A}^{-1}$ is compact.
\end{lemma}
\begin{proof}
Let $u=(\underline{u},\dot{\underline{z}})^\top\in\mathcal
D(\mathcal{A})$, then by integration by parts, we have \begin{eqnarray*} \langle
\mathcal{A}u,
u\rangle_{\mathcal{H}}&=&\sum_{j=0}^N\int_{\O_j}\(-(\s_{j}(x)u_j')'+q_{j}(x)u_j\)u_j dx+\sum_{j=1}^{N}\(\s_{j-1}u_{j-1}'(\ell_{j})-\s_{j}u_{j}'(\ell_{j})\)u_j(\ell_j), \\
&=&\sum_{j=0}^N\int_{\O_j}\(\s_j(x)|u_j'|^{2}+q_j(x)|u_j|^{2}\)dx.\end{eqnarray*}
Since $\s_j>0$ and $q_j\geq0,$ then $\langle
\mathcal{A}u, u\rangle_{\mathcal{H}}>0$ for $u\not\equiv0$, and
hence, the linear operator $\mathcal{A}$ is positive. Furthermore,
it is easy to show that $Ran(\mathcal{A}-iId)=\mathcal{H}$, and this
implies that $\mathcal{A}$ is selfadjoint. Since the space
$\mathcal{W}$ is continuously and compactly embedded in the space
$\mathcal{H}$, then $\mathcal{A}^{-1}$ is compact in $\mathcal{H}$.
The proof is complete.
\end{proof}

Let us consider the problems determined by the equation
\begin{equation}-(\s_{j}(x)\p_j')'+q_{j}(x)\p_j=
 \la \r_{j}(x)\p_j,~~x\in \overline{\O}_j,~~j=0,...,N,\label{eqss:3}
\end{equation}
and the initial conditions
\begin{align}
 &\p_0(0)=\s_{0}\p_0'(0)-1=0,\label{eqss:4}\\
 &\p_{j}(\ell_{j})=\p_{j-1}(\ell_{j}),&j=1,...,N, \label{eqss:5}\\
&\s_{j}\p_{j}'(\ell_{j})=\s_{j-1}\p_{j-1}'(\ell_{j})-M_{j}\la
\p_{j-1}(\ell_{j}),& j=1,...,N,\label{eqss:6}
\end{align}
and
\begin{align}
 &\p_{N}(L)=\s_{N}\p_{N}'(L)+1=0,\label{eqss:8}\\
&\p_{j-1}(\ell_{j})=\p_{j}(\ell_{j}),&j=1,...,N, \label{eqss:9}\\
&\s_{j-1}\p_{j-1}'(\ell_{j})=\s_{j}\p_{j}'(\ell_{j})+M_{j}\la \p_{j}(\ell_{j}),
&j=1,...,N,\label{eqss:10}
\end{align} respectively.
For each $j=0,...,N$, let $\widehat\varphi_j(x,\la)$ and
 $\widehat\psi_j(x,\la)$ are the unique solutions, up to a multiplicative constant,
of the subproblems determined by Equation \eqref{eqss:3} in
$\overline{\O}_j$, and the initial conditions \begin{equation}
\widehat\varphi_j(\ell_j)-1=\widehat\varphi_j'(\ell_j)=0,~j=0,...,N,
\label{eqss:11} \end{equation} and \begin{equation}
\widehat\psi_j(\ell_j)=\s_j\widehat\psi_j'(\ell_j)-1=0,~j=0,...,N, \label{eqss:12}
\end{equation}
respectively. It is known (e.g., \cite[Chapter 1]{LEVITAN} and \cite[Chapter
1]{TITCHMARSH}), that
$\widehat\varphi_j(x,\la)$ and $\widehat\varphi_j'(x,\la)$
(resp. $\widehat\psi_j(x,\la)$ and $\widehat\psi_j'(x,\la)$)
 are entire functions of $\la$ for each
fixed $x\in\overline{\O}_j,~j=0,...,N$.
\begin{lemma}\label{subunquss} Let us fix $j\in\{0,...,N\}$, and let $f_j(\la)$ and $g_j(\la)$
be two analytic functions. Then, the
subproblem determined by Equation \eqref{eqss:3} in
$\overline{\O}_j$, and the initial conditions \begin{eqnarray}
\p_j(\ell_j)=f_j(\la),~\s_j\p'_j(\ell_j)=g_j(\la)~~~~ (\hbox{or}~
\p_j(\ell_{j+1})=f_j(\la),~\s_j\p'_j(\ell_{j+1})=g_j(\la))\label{eqss:13}\end{eqnarray} has a
unique solution $\p_j(x,\la)$, up to a multiplicative
constant, \begin{equation}
{\p}_j(x,\la)=f_j(\la)\widehat\varphi_j(x,\la)+g_j(\la)\widehat\psi_j(x,\la),~x\in\overline{\O}_j,~
j=0,...,N.\label{eqss:14}\end{equation}
 Furthermore, $\p_j(x,\la)$ and $\p_j'(x,\la)$ are entire
functions of $\la$ for each fixed $x\in\overline{\O}_j$,
$j=0,...,N$.
\end{lemma}
\begin{proof} By \eqref{eqss:11}-\eqref{eqss:12}, the Wronskian $$\Delta_j(\la)=
\widehat\varphi_j\s_j\widehat\psi_j'(\ell_j)-
\widehat\varphi_j'\s_j\widehat\psi_j(\ell_j)=1\neq0,~j=0,...,N,$$
and then, $\widehat\varphi_j(x,\la)$ and $\widehat\psi_j(x,\la)$ are two
linearly independent solutions of Equation \eqref{eqss:3} in $\overline{\O}_j$. This implies
that any solution ${\p}_j(x,\la)$ of the subproblem
\eqref{eqss:3}, \eqref{eqss:13}, can be written in the form
\begin{equation*}{\p}_j(x,\la)=C_1\widehat\varphi_j(x,\la)+
C_2\widehat\psi_j(x,\la),~x\in\overline{\O}_j,~j=0,...,N,\end{equation*}
for some constants $C_j\not=0$, $j=0,1.$
Using this together with the initial conditions
\eqref{eqss:11}-\eqref{eqss:12} and \eqref{eqss:13} , we have
$$ {\p}_j(x,\la)=f_j(\la)\widehat\varphi_j(x,\la)+g_j(\la)\widehat\psi_j(x,\la),~
x\in\overline{\O}_j,~j=0,...,N,$$ is a nontrivial solution of the subproblem
  \eqref{eqss:3}, \eqref{eqss:13}. The uniqueness
of solutions follows from the linearity of the equation \eqref{eqss:3}
together with standard theory of differential equations.
Since $f_j(\la)$ and $g_j(\la)$ are analytic functions, then
from the expression \eqref{eqss:14}, $\p_j(x,\la)$ and $\p_j'(x,\la)$ are entire
functions of $\la$ for each fixed $x\in\overline{\O}_j$.
\end{proof}
\begin{lemma}\label{Unqunss}{\bf(a)} The initial value problem  \eqref{eqss:3}-\eqref{eqss:6} has a unique
solution, up to a multiplicative constant, \begin{eqnarray} \label{eqss:15}
 \underline{\varphi}_N:=\big({\varphi}_j(x,\la)\big)_{j=0}^N,
 ~x\in\overline{\O}_j,~j=0,...,N,
\end{eqnarray} where $\varphi_0(x,\la)$ and $\varphi_j(x,\la),$ $j=1,...,N,$
are the unique solutions (up to a scalar) of the initial value
subproblems determined by Equations \eqref{eqss:3}-\eqref{eqss:4} in $\overline{\O}_0,$ and Equations
\eqref{eqss:3}, \eqref{eqss:5}-\eqref{eqss:6} in $\overline{\O}_j,$
$j=1,...,N,$ respectively. Furthermore,
$\underline{\varphi}_N(x,\la)$ and $\underline{\varphi}_N'(x,\la)$ are entire
functions of $\la$ for each fixed
$x\in\overline{\O}$.\\
{\bf(b)} The initial value problem \eqref{eqss:3}, \eqref{eqss:8}-\eqref{eqss:10} has a unique
solution, up to a multiplicative constant, \begin{eqnarray} \label{eqss:16}
 {\underline{\psi}}_N:=\big({\psi}_j(x,\la)\big)_{j=0}^N,
 ~x\in\overline{\O}_j,~j=0,...,N,
\end{eqnarray} where $\psi_N(x,\la)$ and $\psi_j(x,\la),$ $j=0,...,N-1,$ are
the unique solutions (up to a scalar) of the initial value
subproblems determined by Equations \eqref{eqss:3},\eqref{eqss:8} in
${\overline{\O}}_N,$ and Equations \eqref{eqss:3},
\eqref{eqss:9}-\eqref{eqss:10} in ${\overline{\O}}_j,$
$j=0,...,N-1,$ respectively. Furthermore,
$\underline{\psi}_N(x,\la)$ and $\underline{\psi}_N'(x,\la)$ are
entire functions of $\la$ for each fixed $x\in\overline{\O}$.
\end{lemma}
\begin{proof}
For $j=0$, it is known (e.g., \cite[Chapter
1]{LEVITAN} and \cite[Chapter 1]{TITCHMARSH}), that the initial value subproblem \eqref{eqss:3}-\eqref{eqss:4} has a unique solution $\varphi_0:=\widehat\psi_0(x,\la)$, $x\in\overline{\O}_0$, up to a multiplicative
constant, such that $\widehat\psi_0(x,\la)$ and $\widehat\psi_0'(x,\la)$ are
entire functions of $\la$ for each fixed $x\in\overline{\O}_0$. For $j=1$, let $f_1(\la)=\widehat\psi_0(\ell_1,\la)$ and
$g_1(\la)=\s_{0}(\ell_{1})\widehat\psi_{0}'(\ell_{1},\la)-M_{1}\la
\widehat\psi_{0}(\ell_{1},\la)$. Then by
Lemma \ref{subunquss}, the subproblem determined by Equation \eqref{eqss:3}
in $\overline \O_1$ and the initial conditions $$\varphi_1(\ell_1,\la)=\widehat\psi_{0}(\ell_{1},\la) \hbox{ and }\s_{1}(\ell_{1})\varphi_1'(\ell_{1},\la)=\s_{0}(\ell_{1})\widehat\psi_{0}'(\ell_{1},\la)-M_{1}\la
\widehat\psi_{0}(\ell_{1},\la),$$
has a unique solution $\varphi_1(x,\la)$, up to a scalar, \begin{equation}
\varphi_1(x,\la)=\widehat\psi_{0}(\ell_{1},\la)\widehat\varphi_1(x,\la)+
\Big(\s_{0}(\ell_{1})\widehat\psi_{0}'(\ell_{1},\la)-M_{1}\la
\widehat\psi_{0}(\ell_{1},\la)\Big)\widehat\psi_1(x,\la)
,~x\in\overline{\O}_1,\label{eqss:17}\end{equation}
where $\widehat\varphi_1(x,\la)$ and $\widehat\psi_1(x,\la)$ are the
solutions of the subproblems \eqref{eqss:3}, \eqref{eqss:11} and
\eqref{eqss:3}, \eqref{eqss:12} for $j=1$, respectively. Furthermore,
$\varphi_1(x,\la)$ and $\varphi_1'(x,\la)$ are analytic functions of $\la$ for each fixed $x\in\overline{\O}_1$.
For $j=2$, let $f_2(\la)=\varphi_{1}(\ell_2,\la)$ and
$g_2(\la)=\s_{1}(\ell_{2})\varphi_{1}'(\ell_{2},\la)-M_{2}\la
\varphi_{1}(\ell_{2},\la)$. Again by Lemma \ref{subunquss}, the subproblem determined by Equation \eqref{eqss:3}
in $\overline \O_2$ and the initial conditions $\varphi_2(\ell_2,\la)= f_2(\la) \hbox{ and }\s_{2}(\ell_{2})\varphi_2'(\ell_{2},\la)=g_2(\la),$
has a unique solution $\varphi_2(x,\la)$, up to a scalar, such that
$\varphi_2(x,\la)$ and $\varphi_2'(x,\la)$ are entire
function of $\la$ for each fixed $x\in\overline{\O}_1.$ Moreover,
\begin{equation}
\varphi_2(x,\la)=\varphi_{1}(\ell_2,\la)\widehat\varphi_2(x,\la)+
\Big(\s_{1}(\ell_{2})\varphi_{1}'(\ell_{2},\la)-M_{2}\la
\varphi_{1}(\ell_{2},\la)\Big)\widehat\psi_2(x,\la),~x\in\overline{\O}_2,
\label{eqss:18}
\end{equation}
where
$\widehat\varphi_2(x,\la)$ and $\widehat\psi_2(x,\la)$ are the
solutions of the subproblems \eqref{eqss:3}-\eqref{eqss:11} and
\eqref{eqss:3}-\eqref{eqss:12} for $j=2$, respectively. Now, for each $j=3,...,N$, let $$f_j(\la)=\varphi_{j-1}(\ell_j,\la),~ g_j(\la)=\s_{j-1}(\ell_{j})\varphi_{j-1}'(\ell_{j},\la)-M_{j}\la
\varphi_{j-1}(\ell_{j},\la),$$ and iterating Lemma \ref{subunquss}. Then for each $j$, the subproblem determined by Equation \eqref{eqss:3} in $\overline \O_j$ and the initial conditions $$\varphi_j(\ell_j,\la)=\varphi_{j-1}(\ell_{j},\la) \hbox{ and }\s_{j}(\ell_{j})\varphi_j'(\ell_{j},\la)=\s_{j-1}(\ell_{j})\varphi_{j-1}'(\ell_{j},\la)-M_{j}\la
\varphi_{j-1}(\ell_{j},\la),$$
has a unique solution $\varphi_j(x,\la)$, up to a scalar, such that $\varphi_j(x,\la)$ and
 $\varphi_j'(x,\la)$ are analytic
functions of $\la$ for each fixed $x\in\overline{\O}_j.$ Consequently, we have the following iteration formula: for each $j=2,...,N$, and $x\in\overline{\O}_j$,
\begin{equation}\varphi_j(x,\la)=\varphi_{j-1}(\ell_{j},\la)\widehat\varphi_j(x,\la)+
\Big(\s_{j-1}(\ell_{j})\varphi_{j-1}'(\ell_{j},\la)-M_{j}\la
\varphi_{j-1}(\ell_{j},\la)\Big)\widehat\psi_j(x,\la),\label{eqss:19}\end{equation}
where $\widehat\varphi_j(x,\la)$ and
$\widehat\psi_j(x,\la)$ are the solutions of the initial value
subproblems \eqref{eqss:3}-\eqref{eqss:11} and \eqref{eqss:3}-\eqref{eqss:12} for $j=2,...,N$,
respectively. Therefore, the function
\begin{equation} \label{eqss:20}\underline{\varphi}_N:=\big({\varphi}_j(x,\la)\big)_{j=0}^N,
 ~x\in\overline{\O}_j,~j=0,...,N,
\end{equation}
is a nontrivial solution
of Problem  \eqref{eqss:3}-\eqref{eqss:6}. Since,
$\varphi_j(x,\la)$ and $\varphi_j'(x,\la)$ are analytic
functions of $\la$ for each fixed $x\in\overline{\O}_j,$ then by \eqref{eqss:20},
$\underline{\varphi}_N(x,\la)$ and $\underline{\varphi}_N'(x,\la)$
are also entire functions with respect to
$\la$ for each fixed
$x\in\overline{\O}$. We now prove the uniqueness of solutions. Let
$\underline{\varphi}_N^1:=\big({\varphi}_j^1(x,\la)\big)_{j=0}^N$ and
$\underline{\varphi}_N^2:=\big({\varphi}_j^2(x,\la)\big)_{j=0}^N$ are
two linearly independent solutions of Problem
\eqref{eqss:3}-\eqref{eqss:6}. Then, by the linearity
of the equations \eqref{eqss:3}-\eqref{eqss:6}, the function $$ \underline{\widehat{\varphi}}_N:=\underline{\varphi}_N^1-\underline{\varphi}_N^2=
\big({\varphi}_j^1(x,\la)-{\varphi}_j^2(x,\la)\big)_{j=0}^N,
~x\in\overline{\O}_j,~j=0,...,N,$$ is a nontrivial
 solution of the problem determined by Equations \eqref{eqss:3}, \eqref{eqss:5}-\eqref{eqss:6},
and the initial conditions,
  ${\varphi}_0(0)=\s_0{\varphi}_0'(0)=0.$ From this
  and the uniqueness theorem for the equation \eqref{eqss:3} in $\overline{\O}_0$,
   we get ${\varphi}_0^1(x,\la)-{\varphi}_0^2(x,\la)\equiv0$, and then,
 by  \eqref{eqss:5}-\eqref{eqss:6}, one has $${\varphi}_1(\ell_1)=\s_2{\varphi}_1'(\ell_1)=0.$$ Again by
 the uniqueness theorem for the equation \eqref{eqss:3} in $\overline{\O}_1$, ${\varphi}_1^1(x,\la)-{\varphi}_1^2(x,\la)
  \equiv0$.  Iterating this argument, one obtains
  $${\varphi}_j^1(x,\la)={\varphi}_j^2(x,\la),~x\in\overline{\O}_j,~j=0,...,N, $$
a contradiction. The second statement of the Lemma can
be proved in a same way.
\end{proof}

We now prove an asymptotic formula for the solution
 $\underline{\varphi}_N(x, \la)$ of Problem \eqref{eqss:3}-\eqref{eqss:6}. Hereafter, we use these notations
\begin{equation}
\xi_j(x)=\left({\r_{j}(x)\s_{j}(x)}\right)^{-\frac{1}{4}},~
 {\xi_j^*}=\xi_j(\ell_j)\xi_j(\ell_{j+1}),~\Upsilon_j=\prod_{k=0}^{j}
\xi^*_{k},\label{eqss:21}\end{equation}
and \begin{equation} \o_j(x)=\int_{\ell_j}^{x}
\sqrt{\frac{\r_{j}(t)}{\s_{j}(t)}}dt, ~\o_j^*=\o_j(\ell_{j+1}),~\hbox{ and  }
\gamma=\sum_{j=0}^{N}\o_j^*,~x\in\overline{\O}_j,~j=0,...,N.\label{eqss:22} \end{equation}
One has:
\begin{proposition}\label{Aymptotics1x}
Let $\la=\nu^2,$ and let
$\underline{\varphi}_N(x, \la)$ be the solution of Problem
 \eqref{eqss:3}-\eqref{eqss:6}  constructed in
Lemma \ref{Unqunss}.
Then, for each $j=2,...,N$, and every $x\in\overline{\O}_j$,
\begin{align}\label{eqss:23}
\frac{(-1)^{j}\varphi_j(x,\la)}{\ds \prod_{k=1}^{j-1}M_{k}\Upsilon_{j-1}\xi_j(\ell_j)\xi_j(x)}&= M_j\nu^{j-1} \prod_{j=0}^{j-1}
{\sin(\nu\o_j^*)}{\sin(\nu\o_j(x))}[1]-\nu^{j-2}\prod_{k=0}^{j-2}{\sin(\nu\o_k^*)}[1]
 \\\nonumber
&\times\(
\frac{\cos(\nu\o_{j-1}^*)\sin(\nu\o_j(x))}{\xi_{j-1}^2(\ell_j)}+
\frac{\sin(\nu\o_{j-1}^*)\cos(\nu\o_j(x))}{\xi_j^2(\ell_j)}
\)[1]                                                             \end{align}
and \begin{align} \label{eqss:24} \textstyle\dfrac{(-1)^{j}\xi_j\s_{j}\varphi_j'(x,\la)}
{\prod_{k=1}^{j-1}M_{k}\Upsilon_{j-1}\xi_j(\ell_j)}&= M_{j}\nu^j\prod_{k=0}^{j-1}
{\sin(\nu\o_k^*)}{\cos(\nu\o_j(x))}[1]-\nu^{j-1}\prod_{k=0}^{j-2}{\sin(\nu\o_k^*)}[1] \\\nonumber
&\times\(
\frac{\cos(\nu\o_{j-1}^*)\cos(\nu\o_j(x))}{\xi_{j-1}^2(\ell_j)}-
\frac{\sin(\nu\o_{j-1}^*)\sin(\nu\o_j(x))}{\xi_j^2(\ell_j)}
\)[1]                                                              \end{align}
 where  $[1]=1+\mathcal{O}\(\frac{1}{|\nu|}\)$, and ${\varphi}_j(x, \la)$ are given in \eqref{eqss:15}.
\end{proposition}
\begin{proof}
It is known (e.g., \cite[Chapter
1]{LEVITAN} and \cite[Chapter 1]{TITCHMARSH}), that the solutions $\widehat\varphi_j(x,\la)$
and $\widehat\psi_j(x,\la)$ of the subproblems determined by Equation \eqref{eqss:3} in
$\overline{\O}_j$, and the initial conditions \eqref{eqss:11} and \eqref{eqss:12}, satisfy respectively the asymptotics\begin{equation}
\begin{cases}\label{eqss:25}
\widehat\varphi_j(x,\la)= \xi_j(x)\dfrac{\cos(\sqrt{\la}\o_j(x))}{\xi_j(\ell_j)}[1],~j=0,...,N,\\
\s_j(x)\widehat\varphi_j'(x,\la)=- \dfrac{\sqrt{\la}\sin(\sqrt{\la}\o_j(x))}{\xi_j(\ell_j)\xi_j(x)}[1], ~j=0,...,N,
\end{cases}\end{equation} and
\begin{equation}\begin{cases}\label{eqss:26}
\widehat\psi_j(x,\la)=\xi_j(\ell_j)\xi_j(x)\dfrac{\sin(\sqrt{\la}\o_j(x))}
{\sqrt{\la}}[1],~j=0,...,N,\\
\s_j(x){\widehat\psi}_j'(x,\la)=\xi_j(\ell_j) \dfrac{\cos(\sqrt{\la}\o_j(x))}{\xi_j(x)}[1],~j=0,...,N,
\end{cases}\end{equation}
$\hbox{ as } |\la|\rightarrow \infty,$ where
$[1]=1+\mathcal{O}\(\tfrac{1}{\sqrt{|\la|}}\)$. Let $j=1$, then from the expression \eqref{eqss:17} and the asymptotes \eqref{eqss:26} for $j=0$, we have \begin{eqnarray*}
{\varphi_1(x,\la)}=-M_{1}\xi_0^*\nu
\sin(\nu\o_0^*)
\widehat\psi_1(x)[1]+\xi_0^*
\(\frac{\cos(\nu\o_0^*)}{\xi_0^2(\ell_1)}\widehat\psi_1(x)
+\widehat\varphi_1(x)\frac{\sin(\nu\o_0^*)}{\nu}\)[1],~x\in\overline{\O}_1,\end{eqnarray*}
$\hbox{ as } |\nu|\rightarrow \infty$, where $\la=\nu^2$, the quantities $\xi_j^*$ and $\o_j^*$ are given by \eqref{eqss:21} and  \eqref{eqss:22}, respectively. Using this and \eqref{eqss:25}-\eqref{eqss:26} for $j=1$,
a straightforward
calculation gives the following asymptotics \begin{align}\label{eqss:27}
\frac{\varphi_1(x,\la)}{\xi_0^*\xi_1(\ell_1)\xi_1(x)}&=-M_{1}
{\sin(\nu\o_0^*)}{\sin(\nu\o_1(x))}[1]+\\\nonumber
&~~~~\frac{1}{\nu}
\(\frac{\cos(\nu\o_0^*)\sin(\nu\o_1(x))}{\xi_0^2(\ell_1)}+
\frac{\sin(\nu\o_0^*)\cos(\nu\o_1(x))}{\xi_1^2(\ell_1)}\)[1]
                                \end{align}
                                and
 \begin{align}\label{eqss:28}
\frac{\xi_1\s_{1}\varphi_1'(x,\la)}{\xi_0^*\xi_1(\ell_1)}&=-M_{1}\nu
{\sin(\nu\o_0^*)}{\cos(\nu\o_1(x))}[1]+\\\nonumber
&~~~~~~\(\frac{\cos(\nu\o_0^*)\cos(\nu\o_1(x))}{\xi_0^2(\ell_1)}
-
\frac{\sin(\nu\o_0^*)\sin(\nu\o_1(x))}{\xi_1^2(\ell_1)}\)[1].
 \end{align} In particular, with the convention $\prod_{1}^{0}=\prod_{0}^{-1}=1$, the asymptotes  \eqref{eqss:23}-\eqref{eqss:24} hold for $j=1$.
Now, let $j=2$, then by \eqref{eqss:27}-\eqref{eqss:28},  \begin{eqnarray}\label{eqss:29}\left\{
 \begin{array}{ll}
\varphi_1(\ell_2,\la)=-M_{1}\prod_{j=0}^{1}\xi_j^*{\sin(\nu\o_j^*)}[1],\\
\s_{1}\varphi_1'(\ell_2,\la)=-M_{1}\nu
\prod_{j=0}^{1}\xi_j^*{\sin(\nu\o_0^*)}
\dfrac{\cos(\sqrt{\la}\o_1^*)}{\xi_1^2(\ell_2)} [1],
\end{array}
 \right.
\end{eqnarray} and hence, from the expression \eqref{eqss:18}, one gets
 \begin{align*} \frac{\varphi_2(x,\la)}{M_{1}\prod_{j=0}^{1}\xi_j^*}=&
M_{2}\nu^2\prod_{j=0}^{1}
{\sin(\nu\o_j^*)}\widehat\psi_2(x)[1]-\bigg(
\nu\sin(\nu\o_0^*)\frac{\cos(\nu\o_1^*)}{\xi_1^2(\ell_2)}
\widehat\psi_2(x)
+\widehat\varphi_2(x)\prod_{j=0}^{1}{\sin(\nu\O_j^*)}\bigg)[1].\end{align*}
From this together with \eqref{eqss:25}-\eqref{eqss:26} for $j=2$, it follows  \begin{align*}
\dfrac{\varphi_2(x,\la)}{M_{1}\prod_{j=0}^{1}\xi_j^*\xi_2(\ell_2)\xi_2(x)}=& M_2\nu \prod_{j=0}^{1}
{\sin(\nu\o_j^*)}{\sin(\nu\o_2(x))}[1]- \sin(\nu\o_0^*)[1]
 \nonumber\\
&\times\(\frac{\cos(\nu\o_1^*)\sin(\nu\o_2(x))}{\xi_1^2(\ell_2)}+
\frac{\sin(\nu\o_1^*)\cos(\nu\o_2(x))}{\xi_2^2(\ell_2)}\)[1],\\                                                              \frac{\xi_2\s_{2}\varphi_2'(x,\la)}{M_{1}\prod_{j=0}^{1}\xi_j^*\xi_2(\ell_2)}=& \nu^2M_{2}\prod_{j=0}^{1}
{\sin(\nu\o_j^*)}{\cos(\nu\o_2(x))}[1]-
\nu\sin(\nu\o_0^*)[1] \nonumber\\
&\times\(\frac{\cos(\nu\o_1^*)\cos(\nu\o_2(x))}{\xi_1^2(\ell_2)}
-
\frac{\sin(\nu\o_1^*)\sin(\nu\o_2(x))}{\xi_2^2(\ell_2)}\)[1],                                                             \end{align*}
and this implies that \eqref{eqss:23}-\eqref{eqss:24} hold for $j=2$. For each $j=3,...,N$,  following the same argument as above,  by using \eqref{eqss:25}-\eqref{eqss:26} and the iteration formula \eqref{eqss:19},  we get  the asymptotic formulas \eqref{eqss:23}-\eqref{eqss:24}. The proof is complete.
\end{proof}

\begin{theorem}\label{simple}
The eigenvalues $(\la_{n})_{n\in\mathbb{N}^{*}}$ of System
$\(\mathcal{P}_N\)$ \eqref{eqs:8} are simple and constitute a sequence
of positive real numbers:
$$0<\la_{1}<\la_{2}<.......<\la_{n}<.....
\underset{n\rightarrow +\infty}{\longrightarrow}+\infty.$$
The corresponding eigenfunctions \begin{equation}\label{eqss:30}
\({\Phi}_n\)_{n\in\N^*}:=\(\big({\varphi}_j(x,\la_n)\big)_{j=0}^N,
\big({\varphi}_j(\ell_j,\la_n)\big)_{j=1}^N\)_{n\in\N^*}^\top,
~x\in\overline{\O}_j,~j=0,...,N,\end{equation}
can be chosen to constitute an orthogonal basis of $\mathcal{H}$
with the inner product \eqref{eqs:6}, where
${\varphi}_j(x, \la)$, $j=0,...,N,$  are given by \eqref{eqss:15}. Moreover, ${\varphi}_j(x, \la_n),~j=2,...,N,$ satisfy the asymptotes \eqref{eqss:23}-\eqref{eqss:24} for $\la=\la_n$.
\end{theorem}
\begin{proof} It follows from Lemma \ref{operator}, that the spectrum of the linear
operator $\mathcal{A}$ is positive and discrete. Since $\mathcal{A}$
is self-adjoint in ${{\mathcal H}}$, then by Lemma \ref{Unqunss}, all
the eigenvalues $(\la_{n})_{n\in\mathbb{N}^{*}}$ of Problem
$\(\mathcal{P}_N\)$ \eqref{eqs:8} are algebraically simple. By the last condition of  \eqref{eqs:8} and Equations \eqref{eqss:3}-\eqref{eqss:6}, the eigenvalues $(\la_{n})_{n\in\mathbb{N}^{*}}$ of Problem
$\(\mathcal{P}_N\)$ \eqref{eqs:8} are solutions of the equation \begin{equation*} \underline{\varphi}_N(\ell_{N+1},\la)={\varphi}_N(L,\la)=0, \end{equation*}
where $\underline{\varphi}_N(x,\la)$ is defined in Lemma \ref{Unqunss}.   This implies that the corresponding
eigenfunctions $\big(\underline{\p}_{n}\big)_{n\in\N^*}$ of
Problem $\(\mathcal{P}_N\)$ \eqref{eqs:8} have the unique form, up to a
scalar, \begin{equation}
\underline{\p}_{n}:=\big({\varphi}_j(x,\la_n)\big)_{j=0}^N, ~x\in\overline{\O}_j,~j=0,...,N,~n\in\N^*,\label{eqss:31}\end{equation}
where ${\varphi}_j(x, \la)$ are given by \eqref{eqss:15}. Thus, the expression \eqref{eqss:30} follows from \eqref{eqss:2} and \eqref{eqss:31}, which ends the proof of the theorem.
\end{proof}
\subsection{Well-posedness}\label{Sec4-1b}In order to study the well-posedness of the heat model \eqref{eqs:1}-\eqref{eqs:3},
we apply the semigroup theory. Let us consider the following nonhomogeneous problem with zero boundary
conditions
\begin{eqnarray*}\left\{\begin{array}{ll}
\(\partial_t v_j-\partial_x\(\s_{j}(x)\partial_x v_j\)+q_{j}(x)v_j\)(t,x)=f_j(t,x),&t>0,~x\in\O_j,~j=0, ...,N, \\
v_{j-1}(t, \ell_j)=z_j(t)=v_{j}(t, \ell_j),&t>0,~j = 1, ...,N, \\
\(\s_{j}(\ell_j)\partial_x v_j-\s_{j-1}(\ell_j)\partial_x v_{j-1}\)(t,\ell_j)= M_j\partial_t z_{j}(t)+g_j(t),&t>0,~j = 1, ...,N, \\
v_{0}(t,\ell_{0})=v_{0}(t,0)=0,~
v_{N}(t,\ell_{N+1})=v_{N}(t,L)=0,&t>0,
\end{array}\right.\end{eqnarray*}
and initial conditions at $t=0$ \begin{eqnarray*} \left\{
\begin{array}{ll}
  v_j(0,x)=v^{0}_j,~x\in\O_j,&j=0,...,N, \\
  z_j(0)=z^{0}_j,&j = 1, ...,N.
\end{array}
 \right.
\end{eqnarray*}
By letting $V=\(\(v_j\)_{j=0}^{N},\(z_j\)_{j=1}^{N}\)^{\top}$ and
$F=\(\(f_j\)_{j=0}^{N},\(g_j\)_{j=1}^{N}\)^{\top}$, the above problem  can be
rewritten in the abstract Cauchy problem
\begin{equation}\begin{cases}\label{eqref4s3}
  \partial_t{V}(t)+\mathcal{A}V(t)=F(t,x),~t\in\(0,\infty\), \\
  V(0)=V^{0},
\end{cases}\end{equation}
where $\mathcal{A}$ is defined in \eqref{eqss:1} and $V^{0}=\(\(v_j^0\)_{j=0}^{N},
\(z_j^0\)_{j=1}^{N}\)^{\top}$. By virtue  of Lemma
\ref{operator}, $\mathcal{A}$ is an infinitesimal
generator of a strongly continuous semigroup in $\mathcal{H}$. Therefore, from the Lumer-Phillips theorem (e.g., \cite{Pazy}), the
Cauchy problem \eqref{eqref4s3} has a unique mild solution
$V\in
C\([0,T], \mathcal{H}\)$ provided that $V^{0}\in \mathcal{H}$ and $F\in L^1\(\(0,T\);\mathcal{H}\)$. Moreover, if $V^{0}\in \mathcal{D}\(\mathcal{A}\)$ and $F\in C^1\([0,T];\mathcal{H}\)$ then
\eqref{eqref4s3} has a unique classical solution in the space
$C\([0,T], \mathcal{D}\(\mathcal{A}\)\)\cap W^{1,1}\(0,T; \mathcal{H}\).$ If we call $U=\(\(u_j\)_{j=0}^{N},\(z_j\)_{j=1}^{N}\)^{\top}$ the corresponding solution of \eqref{eqs:1}-\eqref{eqs:3}, then the function \begin{align*}
  {V}:=\(\(\(u_j\)_{j=0}^{N-1},u_N-\tfrac{x-\ell_N}{\ell_{N+1}-\ell_N}h(t)\),\(z_j\)_{j=1}^{N}\)^{\top}
\end{align*}
satisfies  \eqref{eqref4s3} with
\begin{equation*}\begin{cases*}
 V^{0}=\(u_0^0,...,u_{N-1}^0,u_N^0-\tfrac{x-\ell_N}{\ell_{N+1}-\ell_N}h(0)\), \\ f_j=0,~j=0,...,N-1,~f_N=-\tfrac{x-\ell_N}{\ell_{N+1}-\ell_N}\(\partial_t h(t)+q_{N}(x)h(t)\),\\
  g_j=0,~j=1,...,N.
\end{cases*}\end{equation*}
Consequently, we have the following well-posedness result for the control system \eqref{eqs:1}-\eqref{eqs:3}.
\begin{proposition}
Let ${U}^0=\(\(u^{0}_j\)_{j=0}^{N},\(z^{0}_j\)_{j=1}^{N}\)^{\top}\in
\mathcal{H}$ and $h(t)\in H^1(0,T)$. Then the
problem \eqref{eqs:1}-\eqref{eqs:3} has a unique solution
$$U=\(\(u_j\)_{j=0}^{N},\(z_j\)_{j=1}^{N}\)^{\top}\in
C([0,T], \mathcal{H}).$$
Moreover, if $U^0\in {\mathcal D}(\mathcal{A})$, then $U\in C\([0,T], \mathcal{D}\(\mathcal{A}\)\)\cap C^1\([0,T], \mathcal{H}\).$
\end{proposition}
For the Schr\"{o}dinger model \eqref{Scontrolanis}, we recall the following well-posedness result (see \cite{Hansen}).
\begin{proposition}
Let ${U}^0:=\(\underline{u}^0 = (u^{0}, u_1^0),z^{0}\)^{\top}\in
H^{-1}(0,1) \times \mathbb{C}$ and $h(t)\in L^2(0,T)$. Then,
Problem \eqref{Scontrolanis} has a unique weak solution (by transposition), $$ {U}:=\(\underline{u} = (u_0, u_1),z\)^{\top}\in C\([0,T],H^{-1}(0,1) \times \mathbb{C}\).$$
\end{proposition}
\section{Spectrum} \label{Sect3}
\setcounter{equation}{0}
In this section, we investigate the main properties of all the eigenvalues $(\la_{n})_{n\in\mathbb{N}^{*}}$ of Problem
$\(\mathcal{P}_N\)$ \eqref{eqs:8}. 
\subsection{The characteristic equation}\label{Sec3-1}
In this subsection, we implement the Wronskian technique (e.g., \cite[Chapter
1]{Freiling} and \cite[Chapter 1]{TITCHMARSH}), to obtain the characteristic equation
for the eigenvalues  $(\la_{n})_{n\in\mathbb{N}^{*}}$ of Problem $\(\mathcal{P}_N\)$ \eqref{eqs:8}. Namely, we enunciate the following result:
\begin{theorem} \label{wronskian}  $\la$ is an eigenvalue of Problem $\(\mathcal{P}_N\)$ \eqref{eqs:8}, if and only if, the Wronskians \begin{equation*}
{\Delta}_{j}(\la)=\varphi_{j}(x,\la)\s_{j}(x)\psi'_{j}(x,\la)-\varphi'_{j}(x,\la)
\s_{j}(x)\psi_{j}(x,\la)=0,~\forall x\in\overline{\O}_{j},~\forall j\in\{0,...,N\}, \end{equation*}
where $\varphi_{j}(x,\la)$ and $\psi_{j}(x,\la)$  are defined by \eqref{eqss:15} and \eqref{eqss:16}, respectively.
\end{theorem}
For the proof of this theorem, we need the following remarkable and useful property of the wronskians ${\Delta}_{j}(\la)$.
\begin{lemma}\label{Lem2decroicarv} One has:\begin{equation} \Delta_j(\la)=\Delta_k(\la),~~\forall
j,k\in\{0,...,N\} \hbox{ with } j\not=k.\label{eqsss:10}\end{equation}
\end{lemma}
\begin{proof}
By the initial conditions
\eqref{eqss:5}-\eqref{eqss:6} and \eqref{eqss:9}-\eqref{eqss:10}, one gets  \begin{eqnarray*}
\Delta_{j-1}(\la)&=&\varphi_{j-1}\s_{j-1}\psi_{j-1}'(\ell_j)-
\psi_{j-1}\s_{j-1}\varphi_{j-1}'(\ell_j),\\
&=&\varphi_j(\ell_j)\(\s_j\psi_j'(\ell_j)+M_j\la\psi_j(\ell_j)\)-
\psi_j(\ell_j)\(\s_j\varphi_j'(\ell_j)+M_j\la\varphi_j(\ell_j)\),\\
&=&\Delta_j(\la),~j=1,...,N.
 \end{eqnarray*}
By Equations \eqref{eqss:3}, $\partial_x\Delta_j(\la)=0$, for
all $j\in\{0,...,N\}$. Thus, from the above, \eqref{eqsss:10} follows.
\end{proof}

\begin{proof}[Proof of Theorem \ref{wronskian}.]
We first argue by contradiction, so let
$\{\la,\underline{\p}(x,\la)\}$ be an eigenpair of Problem
$\(\mathcal{P}_N\)$ \eqref{eqs:8} and suppose that
${\Delta}_{j^*}(\la)\neq0$ for some $j^*\in\{0,...,N\}$. Under this assumption together with Lemma \ref{Lem2decroicarv}, it follows
\begin{equation}
\Delta_j(\la)\neq0,~j=0,...,N,\label{g10}\end{equation} and this implies that, $\varphi_j(x,\la)$ and $\psi_j(x,\la)$ are linearly
independent solutions of Equation \eqref{eqss:3} in each the
subintervals $\overline{\O}_j$, $j=0,...,N$. Consequently, any solution of the problem determined by Equations \eqref{eqss:3},
\eqref{eqss:9}-\eqref{eqss:10} may be
expressed as a linear combination of $\underline{\varphi}_N(x,\la)$ and
$\underline{\psi}_N(x,\la)$ for $x\in\overline{\O}$. Therefore the eigenfunction
$\underline{\p}(x,\la)$ of Problem
$\(\mathcal{P}_N\)$ \eqref{eqs:8} can be written in the form \begin{equation}
{\underline{\p}}:=\big(
C_j\varphi_j(x,\la)+{\widehat{C}_j}\psi_j(x,\la)\big)_{j=0}^N,~ x\in\overline{\O}_j,~j=0,...,N,\label{eqsss:1}\end{equation} for some  constants $C_j$ and $\widehat{C}_j.$ Substituting this expression into the first
two conditions of System $\(\mathcal{P}_N\)$ \eqref{eqs:8}, we get \begin{equation}
C_{j-1}\varphi_{j-1}(\ell_j)+\widehat{C}_{j-1}\psi_{j-1}(\ell_j)
=C_j\varphi_j(\ell_j)+{\widehat{C}_j}\psi_j(\ell_j),~
j=1,...,N,\label{eqsss:2}\end{equation} and \begin{eqnarray} \label{eqsss:3}M_{j}\la
\(C_{j}\varphi_{j}(\ell_j)+\widehat{C}_{j}\psi_{j}(\ell_j)\)&=&
\s_{j-1}\(C_{j-1}\varphi_{j-1}'(\ell_j)+\widehat{C}_{j-1}\psi_{j-1}'(\ell_j)\)\\\nonumber
&~&-\s_{j}\(C_j\varphi_j'(\ell_j)+{\widehat{C}_j}\psi_j'(\ell_j)\),~j=1,...,N.\end{eqnarray}
Using Conditions \eqref{eqss:5}-\eqref{eqss:6} and \eqref{eqss:9}-\eqref{eqss:10} in \eqref{eqsss:2}-\eqref{eqsss:3}, one has \begin{eqnarray} \(C_j-C_{j-1}\)\varphi_j(\ell_j)+\({\widehat{C}_j}-\widehat{C}_{j-1}\)
\psi_j(\ell_j)=0,&~j=1,...,N,\label{eqsss:4}\\
\(C_j-C_{j-1}\)\s_{j}\varphi_j'(\ell_j)+
\({\widehat{C}_j}-\widehat{C}_{j-1}\)\s_{j}\psi_j'(\ell_j)=0,&~j=1,...,N.
\label{eqsss:5}\end{eqnarray} By multiplying Equations
\eqref{eqsss:4} and \eqref{eqsss:5}, respectively, by
$\s_{j}\varphi_j'(\ell_j)$ and $\varphi_j(\ell_j)$, a simple
calculations yields $$
\(\widehat{C}_{j-1}-\widehat{C}_{j}\)\(\varphi_j\s_{j}\psi_j'(\ell_j)-\varphi_j'\s_{j}\psi_j(\ell_j)\)=
\(\widehat{C}_{j-1}-\widehat{C}_{j}\)\Delta_j(\la)=0,~
j=1,...,N.$$ Similarly, $$
\(C_{j}-C_{j-1}\)\(\varphi_j\s_{j}\psi_j'(\ell_j)-\varphi_j'\s_{j}\psi_j(\ell_j)\)=
\(C_{j}-C_{j-1}\)\Delta_j(\la)=0,~
j=1,...,N.$$ From
the above together with \eqref{g10}, one has \begin{equation} {C}_{j}=C_{j-1}
\hbox {  and  } \widehat{C}_{j-1}=\widehat{C}_{j},~j=1,...,N. \label{eqsss:6}\end{equation} From \eqref{eqss:4} and \eqref{eqss:8},  \begin{equation} \Delta_1(\la)=-\s_{1}(0)\psi_1(0,\la) \hbox{ and }
\Delta_N(\la)=\s_{N}(L)\varphi_N(L,\la).\label{eqsss:7}\end{equation}
Substituting \eqref{eqsss:1} into the last condition of \eqref{eqs:8} and
using \eqref{eqsss:7}, one gets $$ \widehat{C}_1
\psi_1(0,\la)=0=-\widehat{C}_1\frac{\Delta_1(\la)}{\s_{1}(0)}
\hbox{ and } C_N
\varphi_N(L,\la)=0=C_N\frac{\Delta_N(\la)}{\s_{N}(L)}. $$ From
this together with \eqref{eqs:5}, \eqref{g10} and \eqref{eqsss:6}, we get
$C_j=\widehat{C}_{j}=0$ for all $j\in\{0,...,N\}$. Thus from
\eqref{eqsss:1}, $\underline{\p}(x,\la)=0$, a contradiction.
Reciprocally, if ${\Delta}_{j}(\la)=0$ for all $j\in\{0,...,N\}$, then $\Delta_N(\la)=0$. This implies that,  \begin{equation}
\varphi_N(x,\la)=C\psi_N(x,\la),~x\in\overline{\O}_N,\label{eqsss:8}\end{equation} for some
constant $C\neq0$, where $\varphi_N(x,\la)$ and $\psi_N(x,\la)$ are defined by
\eqref{eqss:15} and \eqref{eqss:16}, respectively. Since
$\psi_N(L,\la)=0$, then from \eqref{eqsss:8}, the
 solution $\underline{\p}_N(x,\la)$ of Problem $\(\mathcal{P}_N^0\)$ \eqref{eqss:3}-\eqref{eqss:6}
 satisfies the boundary condition $\underline{\p}_N(L,\la)=0$. Thus by Theorem \ref{simple}, $\{\la, \underline{\p}_N(x,\la)\}$ is an eigenpair of Problem
$\(\mathcal{P}_N\)$ \eqref{eqs:8}.
The Theorem is proved.
\end{proof}
\subsection{Interlacing of eigenvalues and Weyl's formula }\label{Sec3-2}
In this subsection, we prove that all the eigenvalues $(\lambda_{n})_{n\in\N^*}$ of Problem $\(\mathcal{P}_{N}\)$ \eqref{eqs:8} interlace those of the $N+1$ decoupled
rods with homogeneous Dirichlet boundary conditions. As consequence, we establish the Weyl's formula \eqref{Eqs:7}. Set
\begin{equation}\Xi_N:=\{\mu_n^{N,D}\}_{1}^{\infty}=\bigcup_{j=0}^{N}
\{\widehat\mu_n^{j,D}\}_{1}^{\infty},\label{eqs:9}\end{equation}
where $\(\widehat\mu_n^{j,D}\)_{n\in\N^*},~
j=1,...,N,$ are the eigenvalues of the $N+1$ Direchlet subproblems
 \begin{eqnarray} \label{eqs:10}\left\{\begin{array}{ll}
 -(\s_{j}(x)\p)'+q_{j}(x)\p=\la \r_{j}(x)\p,~~x\in\O_j,& j=0,...,N, \\
\p(\ell_j)=\p(\ell_{j+1})=0.
 \end{array}\right.
 \end{eqnarray}
One has:
\begin{theorem}\label{interlcing} The eigenvalues $(\la_{n})_{n\in\mathbb{N}^{*}}$ and
$\(\mu_n^{N,D}\)_{n\in\N^*}$ interlace in the following sense:
\begin{equation} 0<\la_{1}\leq\mu^{N,D}_{1}~~\hbox{ and }~~\mu_{n}^{N,D}
\leq\la_{n+1}\leq\mu^{N,D}_{n+1},~\forall n\in\N^*.
 \label{eqs:11}
\end{equation} Moreover, the eigenvalues $(\la_{n})_{n\in\mathbb{N}^{*}}$  satisfy the Weyl's type asymptotic
formula: \begin{equation}
\label{Eqsss:29} \lim_{n\to\infty}\frac{\la_{n}}{n^2\pi^2}=
{\(\sum_{j=0}^{N}\int_{\Omega_j}
\sqrt{\frac{\r_j(x)}{\s_j(x)}}dx\)^{-2}}.\end{equation}
\end{theorem}
To this end, let
$\(\widehat{\la}_n\)_{n\in\N^*}$ denote the eigenvalues of the
Problem $\(\mathcal{P}_{N-1}\)$ \eqref{eqs:8}(i.e., Problem \eqref{eqs:8}
with $N-1$ point masses). By virtue of Theorem \ref{simple}, all the
eigenvalues $\(\widehat{\la}_n\)_{n\in\N^*}$ are positive and simple: \begin{equation}
0<\widehat\la_{1}<\widehat\la_{2}<.......<\widehat\la_{n}<.....\underset{n\to
\infty}{\to}\infty .\label{eqsss:11}\end{equation}  Let
$\underline{\p}_{N}(x,\la)$ be the solution of the initial
problem \eqref{eqss:3}-\eqref{eqss:6} constructed in
Lemma \ref{Unqunss}. Let us introduce the variable complex function
\begin{equation}
F_{N-1}(\la)=\dfrac{\s_{N-1}(\ell_N)\varphi_{N-1}'(\ell_N,\la)}{\varphi_{N-1}(\ell_N,\la)},
~\la\in\(-\infty,
\widehat{\la}_{1}\)\bigcup\left\{\bigcup_{n=0}^\infty
\(\widehat{\la}_n,\widehat{\la}_{n+1}\)\right\},
\label{eqsss:12}\end{equation} where $\varphi_{N-1}(x,\la)$ is the restriction of
$\underline{\p}_{N}(x,\la)$ to the subinterval
$\overline{\O}_{N-1}$. From \eqref{eqsss:11}, $F_{N-1}(\la)$ is well-defined
on all the intervals $\(-\infty,
\widehat{\la}_{1}\)$ and $\(\widehat{\la}_{n}, \widehat{\la}_{n+1}\)$,
$n\in\N^*$. In view of Lemma \ref{Unqunss}, $F_{N-1}(\la)$ is a meromorphic function. Moreover, the
poles of the function $F_{N-1}(\la)$ are the eigenvalues
$\(\widehat{\la}_n\)_{n\in\N^*}$ of Problem $\(\mathcal{P}_{N-1}\)$
\eqref{eqs:8} and their zeros are eigenvalues of the problem \begin{eqnarray*}
\left\{\begin{array}{ll}
 -(\s_{j}(x)\p_j')'+q_{j}(x)\p_j=\la \r_{j}(x)\p_j,~~x\in \O_j,&j=0,...,N-1, \\
\p_{j-1}(\ell_{j})=\varphi_{j}(\ell_{j}),& j=1,...,N-1, \\
\s_{j-1}\varphi_{j-1}'(\ell_{j})-\s_{j}\varphi_{j}'(\ell_{j})=M_{j}\la \p_{j-1}(\ell_{j}),& j=1,...,N-1,\\
\p_0(\ell_0)=\p_0(0)=0,~ \p_{N-1}'(\ell_{N})=0.
 \end{array}\right.\end{eqnarray*}
Similarly, let $\underline{\psi}_N(x,\la)$ be the solution of the
initial problem \eqref{eqss:3}, \eqref{eqss:8}-\eqref{eqss:10} constructed in
Lemma \ref{Unqunss}, and let us consider the meromorphic function \begin{equation}
F_{N}(\la)=\dfrac{\s_{N}(\ell_N)\psi_{N}'(\ell_N,\la)}{\psi_{N}(\ell_N,\la)}
,
~\la\in\(-\infty,
\widehat{\mu}^{N,D}_1\)\bigcup\left\{\bigcup_{n=0}^\infty
\(\widehat{\mu}^{N,D}_n,\widehat{\mu}^{N,D}_{n+1}\)\right\},~n\in\N^*,
\label{eqsss:13}\end{equation}
where $\psi_N(x,\la)$ is the restriction of
$\underline{\psi}_N(x,\la)$ to the subinterval
$\overline{\O}_N$. Obviously, the poles of the function
$F_{N}(\la)$ are the eigenvalues $\(\widehat
\mu_n^{N,D}\)_{n\in\N^*}$ of the Dirichlet problem \eqref{eqs:10} on
the subinterval ${\O}_N$, while, their zeros are
eigenvalues of the problem \begin{eqnarray*} \left\{\begin{array}{ll}
 -(\s_{N}(x)\p)'+q_{N}(x)\p=\la \r_{N}(x)\p,~~x\in {\O}_N, \\
\p'(\ell_N)=\p(\ell_{N+1})=0.
 \end{array}\right.
 \end{eqnarray*}
\begin{proposition}\label{Decroi}
{\bf(a)} $F_{N-1}(\la)$ is a
decreasing function along each of the intervals $\(-\infty,
\widehat{\la}_{1}\)$ and
 $\(\widehat{\la}_{n}, \widehat{\la}_{n+1}\)$, $n\in\N^*$. Furthermore, it decreases from $+\infty$ to $-\infty$.\\
{\bf(b)} $F_{N}(\la)$ is an increasing function from $-\infty$
to $+\infty$ along each of the intervals $\(-\infty, \widehat \mu_{1}^{N,D}\)$ and
 $\(\widehat \mu_{n}^{N,D}, \widehat \mu_{n+1}^{N,D}\)$, $n\in\N^*$.
\end{proposition}
\begin{proof} Let
$\underline{\p}_{N}(x,\la)$ be the solution of the initial
problem $\(\mathcal{P}_{N}^0\)$ \eqref{eqss:3}-\eqref{eqss:6} constructed in
Lemma \ref{Unqunss}. Let us first prove that \begin{eqnarray}{\partial_\la}F_{N-1}(\la) =
\dfrac{-1}{\varphi_{N-1}^2(\ell_N,\la)}\(
\sum_{j=0}^{N-1}\int_{{\O}_j}\varphi_{j}^2(x,\la)dx
+\sum_{j=1}^{N-1}M_{j}\varphi_j^2(\ell_{j},\la)\),\label{eqsss:14}\end{eqnarray}
where $\varphi_j(x,\la)$ are given by \eqref{eqss:15}. To this end, let $\la, \mu\in\(-\infty,
\widehat{\la}_{1}\)$ or $\(\la, \mu\in\(\widehat{\la}_{n},
\widehat{\la}_{n+1}\)\)$, and let us denote by
$$\widehat{\Delta}_j(x)=\varphi_j(x,\la)\s_j(x)\varphi_j'(x,\mu)-\varphi_j'(x,\la)\s_j(x)\varphi_j(x,\mu)
,~x \in\overline{\O}_j,~j=0,...N-1, $$
where $\la\not=\mu$. By
\eqref{eqss:3}, \begin{eqnarray*}\widehat{\Delta}_j'(x)&=&
\varphi_j(x,\la)(\s_{j}\varphi_j')'(x,\mu)-(\s_{j}\varphi_j')'(x,\la)\varphi_j(x,\mu)\\
&=& \(q_j(x)-\mu\r_j(x)\)\varphi_j(x,\la)\varphi_j(x,\mu)-
\(q_j(x)-\la\r_j(x)\)\varphi_j(x,\la)\varphi_j(x,\mu)\\
&=&(\la-\mu)\r_j(x)\varphi_j(x,\la)\varphi_j(x,\mu),~x
\in\overline{\O}_j,~j=0,...N-1, \end{eqnarray*} and this implies
that \begin{equation} \sum_{j=0}^{N-1} \(\widehat{\Delta}_j(\ell_{j+1})
-\widehat{\Delta}_j(\ell_{j})\)=(\la-\mu)\sum_{j=0}^{N-1}
\int_{{\O}_j}\varphi_j(x,\la)\varphi_j(x,\mu)dx.\label{eqsss:15}\end{equation}
From \eqref{eqss:4}, $\widehat{\Delta}_0(\ell_{0})=0$, and then, by
\eqref{eqsss:15}, it follows \begin{equation}\label{eqsss:16} \widehat{\Delta}_{N-1}(\ell_{N})=
\sum_{j=1}^{N-1}\(\widehat{\Delta}_{j}(\ell_{j})-\widehat{\Delta}_{j-1}(\ell_{j})\)+
(\la-\mu)\sum_{j=1}^{N-1}\int_{{\O}_j}\varphi_j(x,\la)\varphi_j(x,\mu)dx.
\end{equation} Using the initial conditions
\eqref{eqss:5}-\eqref{eqss:6}, one gets \begin{eqnarray*}
\widehat{\Delta}_{j-1}(\ell_{j})&=&
\varphi_{j-1}(\ell_{j},\la)\s_{j-1}(\ell_j)\varphi_{j-1}'(\ell_{j},\mu)-\s_{j-1}(\ell_j)\varphi_{j-1}'(\ell_{j},\la)\varphi_{j-1}(\ell_{j},\mu)
\nonumber\\
&=&
\varphi_{j}(\ell_{j},\la)\(\s_{j}\varphi_{j}'+M_{j}\mu\varphi_{j}\)(\ell_{j},\mu)-
\(\s_{j}\varphi_{j}'+M_{j}\la\varphi_{j}\)(\ell_{j},\la)\varphi_i(\ell_{j},\mu)\nonumber\\
&=&\widehat{\Delta}_{j}(\ell_{j})
+M_{j}(\mu-\la)\varphi_{j}(\ell_{j},\la)\varphi_{j}(\ell_{j},\mu),~j=1,...N-1,\end{eqnarray*} and this implies that,
\begin{equation*} \sum_{j=1}^{N-1}
\(\widehat{\Delta}_{j}(\ell_{j})-\widehat{\Delta}_{j-1}
(\ell_{j})\)=(\la-\mu)\sum_{j=1}^{N-1}M_{j}\varphi_{j}(\ell_{j},\la)
\varphi_{j}(\ell_{j},\mu).\end{equation*}
From this together with \eqref{eqsss:16}, one has \begin{alignat}{2}
\label{eqsss:17}{\frac{\widehat{\Delta}_{N-1}(\ell_{N})}{\s_{N-1}(\ell_N)}}=&
\textstyle{\varphi_{N-1}(\ell_N,\la)\dfrac{\varphi_{N-1}'(\ell_N,\mu)-
\varphi_{N-1}'(\ell_N,\la)}{\la-\mu}-
\varphi_{N-1}'(\ell_N,\la)\dfrac{\varphi_{N-1}(\ell_N,\mu)-
\varphi_{N-1}(\ell_N,\la)}{\la-\mu}}\\
=&\dfrac{1}{\s_{N-1}(\ell_N)}\(\sum_{j=0}^{N-1}\int_{\O_{j}}
\varphi_{j}(x,\la)\varphi_{j}(x,\mu)dx
+\sum_{j=1}^{N-1}M_{j}\varphi_j(\ell_{j},\la)\varphi_j(\ell_{j},\mu)\).\nonumber
\end{alignat}
Thus, passing to the limit as $\mu \rightarrow \la$ in
\eqref{eqsss:17} and dividing both sides by
$\varphi_{N-1}^2(\ell_N,\la)$, we get \eqref{eqsss:14}. We now prove that \begin{equation} \lim_{\la
\to-\infty}F_{N-1}(\la ) = + \infty.\label{eqsss:18}\end{equation}
Let $\la=-|\nu|^2$, where $\nu\in\R^*$.
By Lemma \ref{Aymptotics1x}, one has \begin{equation} F_{N-1}(\la )\sim
\frac{i|\nu|\cos\(i|\nu|\o_{N-1}^*\)[1]}{\xi_{N-1}^2(\ell_N)\sin\(i|\nu|\o_{N-1}^*\)[1]},
~\hbox{as}~|\nu|\to\infty,\label{eqsss:19}\end{equation}
where $[1]=1+\mathcal{O}\(\frac{1}{|\nu|}\)$, $\sqrt{i}=-1$ is the imaginary unit,
 $\xi_{N-1}$
and $\omega_{N-1}^*$ are defined by \eqref{eqss:21} and  \eqref{eqss:22}, respectively.
Since, $\sin(i|\nu|) = i \sinh(|\nu|)$ and $\cos(i|\nu|) = \cosh(|\nu|)$, then
by \eqref{eqsss:19}, one gets
$$F_{N-1}(\la )\sim
\frac{|\nu|}{\xi_{N-1}^2(\ell_N)},\hbox{as}~|\nu|\to\infty,$$
and this proves \eqref{eqsss:18}.
On the other hand, the poles
$\(\widehat{\la}_n\)_{n\in\N^*}$ and the zeros of function
$F_{N-1}(\la)$ do not coincide, since otherwise,
$\widehat{\la}_n$ would be an eigenvalue of Problem
$\(\mathcal{P}_{N-1}\)$ \eqref{eqs:8} for which
$\varphi_{N-1}(\ell_N,\widehat{\la}_n)=\varphi_{N-1}'(\ell_N,\widehat{\la}_n)=0$
, a contradiction. From this together with \eqref{eqsss:14}, it follows  \begin{equation*}
\lim_{\la
\to\widehat{\la}^{-}_n}F_{N-1}(\la ) = - \infty   ~~\hbox{ and }
\lim_{\la \to \widehat{\la}^{+}_n}F_{N-1}(\la ) = + \infty,
~n\in\N^*,\end{equation*} and this ends the proof of the statement {\bf (a)} of Proposition \ref{Decroi}. The second statement of the proposition can
be proved in a same way. The proof is complete.
\end{proof}

We are now ready to prove Theorem \ref{interlcing}.
\begin{proof} Let $\(\widehat{\la}_n\)_{n\in\N^*}$ are the eigenvalues the problem
$\(\mathcal{P}_{N-1}\)$ \eqref{eqs:8} (i.e., Problem \eqref{eqs:8}
with $N-1$ point masses). Set $
{\Gamma}:=\{\gamma_n\}_1^\infty=
\left\{\widehat\mu_n^{N,D}\right\}_1^\infty\bigcup\left\{\widehat{\la}_n\right\}_1^\infty,
$
where $\(\widehat
\mu_n^{N,D}\)_{n\in\N^*}$ are the eigenvalues of the Dirichlet subproblem
\eqref{eqs:10} for $x\in{\O}_N$. Since $\(\widehat{\la}_n\)_{n\in\N^*}$ and $\(\widehat
\mu_n^{N,D}\)_{n\in\N^*}$ are simple, then ${\Gamma}$ has a decomposition ${\Gamma}={\Gamma^*}\bigcup{\Gamma^+}$, where \begin{equation}
\Gamma^*:=\{\gamma_n^*,~\hbox{for some }n\in\N^*\}=\{\gamma_n\in\Gamma~:
~\widehat{\la}_j=\widehat\mu_k^{N,D}~~~\hbox{for
some}~j,k\in\mathbb{N^{*}}\} \label{eqsss:20} \end{equation}
and \begin{equation} \Gamma^+:=\{\gamma_n^+\}_1^{\infty}=\Gamma\backslash\Gamma^*~~:~~ 0<\gamma_1^+<\gamma_2^+<.......<\gamma_n^+<.....\underset{n\rightarrow
\infty}{\longrightarrow}\infty. \label{eqsss:21}\end{equation} We prove
\eqref{eqs:11}-\eqref{Eqsss:29} by induction. It is known \cite[Corollary 3.4]{Hedi2018}, that
the eigenvalues
 $(\la_n)_{n\in\N^*}$ of Problem $\(\mathcal{P}_{1}\)$ \eqref{eqs:8} satisfy:
\begin{align} \label{eqsss:22}
0<\la_{1}\leq\mu_{1}^{1,D}~~,~~
 \mu_{n}^{1,D}\leq\la_{n+1}\leq\mu_{n+1}^{1,D},~\forall n\in\N^*,
\end{align} and  \begin{align}\label{eqss::71}
\lim_{n\to\infty}\frac{\la_{n}}{n^2\pi^2}=
{\(\sum_{j=0}^{1}\int_{\O_j}\sqrt{\frac{\r_j(x)}{\s_j(x)}}dx\)^{-2}},
\end{align}
where $\mu_n^{2,D}\in\Xi_1$, and $\Xi_1$ is defined by
$\eqref{eqs:9}$. This means that \eqref{eqs:11}-\eqref{Eqsss:29} hold in the case
$N=1$. Assume that \eqref{eqsss:22}-\eqref{eqss::71} hold for $j\leq N-1$, then by
induction hypothesis, the eigenvalues
 $(\widehat{\la}_{n})_{n\in\N^*}$ of Problem $\(\mathcal{P}_{N-1}\)$ \eqref{eqs:8} satisfy:
 \begin{equation}
 0<\widehat\la_{1}\leq\mu_{1}^{N-1,D}~~,~~\mu_{n}^{N-1,D}
 \leq\widehat\la_{n+1}\leq\mu_{n+1}^{N-1,D},~\forall n\in\N^*,
 \label{eqsss:23}
\end{equation} and \begin{equation} \label{eqsss:43}
\lim_{n\to\infty}\frac{\widehat{\la}_n}{n^2\pi^2}=
{\(\sum_{j=0}^{N-1}\int_{\O_j}\sqrt{\frac{\r_j(x)}{\s_j(x)}}\)^{-2}},\end{equation}

where $\mu_n^{N-1,D}\in\Xi_{N-1},$ and $\Xi_{N-1}$ is defined by
$\eqref{eqs:9}$. First, we prove the interlacing formula \eqref{eqs:11}. In view of Theorem \ref{wronskian},
\begin{equation}\varphi_N(\ell_N,\la_n)\s_N(\ell_N)\psi'_N(\ell_N,\la_n)-
\varphi'_N(\ell_N,\la_n)\s_N(\ell_N)\psi_N(\ell_N,\la_n)=0,~n\in\N^*,\label{eqsss:24}\end{equation}
where $\varphi_N(x,\la)$ and $\psi_N(x,\la)$ are respectively given by \eqref{eqss:15} and \eqref{eqss:16}.
Then, we have only examine the following cases:\\
{\bf{Case} $1.$} If $\varphi_N(\ell_N,\la_n)\not=0$, $n\in\N^*$, then by
\eqref{eqsss:24}, we get $\psi_N(\ell_N,\la_n)\not=0$. This means that
 \begin{equation} \la_n\in\Pi:=\left\{\(-\infty,
\gamma_1^+\)\bigcup\left\{\bigcup_{n=0}^\infty
\(\gamma_n^+,\gamma_{n+1}^+\)\right\},~\gamma_n^+\in
\Gamma^+,~n\in\N^*\right\},\label{eqsss:25}\end{equation}
where $\Gamma^+$ is defined by \eqref{eqsss:21}. From
\eqref{eqss:5}-\eqref{eqss:6} and \eqref{eqsss:24}, one obtains
\begin{alignat}{1}\s_N(\ell_N)\varphi_{N-1}\psi'_N(\ell_N,\la_n)-
\s_{N-1}(\ell_N)\psi_N\varphi_{N-1}'(\ell_{N},\la_n)+M_{N}\la_n
\psi_N\varphi_{N-1}(\ell_{N},\la_n)=0,\label{eqsss:26}
     \end{alignat} which can be rewritten in the form \begin{equation}
F_N(\la_n)-{F}_{N-1}(\la_n)=-M_{N}\la_n,~\la_n\in\Pi,\label{eqsss:27}\end{equation} where $F_{N-1}(\la)$ and $ F_{N}(\la)$ are respectively
given by \eqref{eqsss:12} and \eqref{eqsss:13}. By \eqref{eqsss:25}, we have
  $\gamma_n^+\neq \gamma_{n+1}^{+}$, $n\in\N^*$. By
Proposition \ref{Decroi},
$F_{N}(\la)-{F}_{N-1}(\la)$ is an increasing function
from $-\infty$ to $+\infty$ along each of the intervals $\(-\infty,
\gamma_1^+\)$ and
$\(\gamma_n^+, \gamma_{n+1}^{+}\)$, $n\in\N^*$. Clearly,
that the solution of the equation \eqref{eqsss:27} are the eigenvalues
 $(\la_n)_{n\in\N^*}$ of Problem $\(\mathcal{P}_{N}\)$ \eqref{eqs:8}.
  Moreover, if
 $F_{N}(\la')-{F}_{N-1}(\la')$  for some $\la'$, then $\la'$ is an eigenvalue of the problem $\(\mathcal{P}_{N}\)$ \eqref{eqs:8} for $M_N = 0$.
 Consequently,
 from the curves of the functions
$F_{N}(\la)-{F}_{N-1}(\la)$ and $-M_{N}\la$, one has
\begin{equation} 0< \la_{1}< \la_{1}'<\gamma_{1}^{+}~~\hbox{ and
}~~\gamma_{n}^{+}<\la_{n+1}<\la_{n+1}'<\gamma_{n+1}^{+},~n\in\N^*,\label{eqsss:28}\end{equation}
where  $(\la_n')_{n\in\N^*}$ are the eigenvalues of Problem $\(\mathcal{P}_{N}\)$ \eqref{eqs:8}
 for $M_N = 0$. Since $\gamma_n^+\neq \gamma_{n+1}^{+}$, then by \eqref{eqsss:21}, we have
  $$\gamma_{n}^{+}=\widehat{\la}_j ~\hbox{ and  }~ \gamma_{n+1}^{+}=\widehat
\mu_k^{N,D}~~ (\hbox{or } \gamma_{n}^{+}= \widehat
\mu_k^{N,D} \hbox{ and  } \gamma_{n+1}^{+}=\widehat{\la}_j),~j,k\in\N^*.$$ We may assume
without loss of generality that in \eqref{eqsss:28},
$\gamma_{n}^{+}=\widehat{\la}_j \hbox{ and } \gamma_{n+1}^{+}=\widehat
\mu_k^{N,D}.$ Then, by \eqref{eqs:9}, \eqref{eqsss:23} and \eqref{eqsss:28}, one gets
\begin{equation} 0< \la_{1}< \la_{1}'<\mu_1^{N,D}~~\hbox{ and
}~~\mu_n^{N,D}<\la_{n+1}<\la_{n+1}'<\mu_{n+1}^{N,D},~n\in\N^*,\label{eqsss:45}\end{equation}
with $\mu_n^{N,D}= \mu_{n}^{N-1,D}$
and  $\mu_{n+1}^{N,D}= \widehat\mu_{k}^{N-1,D}.$  This ends the proof of \eqref{eqs:11} in this case.\\
{\bf{Case} $2.$} If $\varphi_N(\ell_N,\la_n)=0$, for some $n\in\N^*$, then by
\eqref{eqsss:24} , we have $\psi_N(\ell_N,\la_n)=0$. This implies $\la_n\in\Gamma^*$, where
$\Gamma^*$ is defined by \eqref{eqsss:20}. Consequently, $\la_n$, $n\in\N^*$,
is simultaneously an eigenvalue of the Dirichlet subproblem \eqref{eqs:10} for $x\in\O_N$ and Problem $\(\mathcal{P}_{N-1}\)$
\eqref{eqs:8}. Thus from \eqref{eqsss:20}
and \eqref{eqsss:23}, one has \begin{equation} \la_{n+1}=\gamma_{n+1}^{*}=\widehat\mu_{k}^{N,D}=\la_{n+1}',
\hbox{ for some } n\in\N^*, \label{eqsss:46}\end{equation}
and then \eqref{eqs:11} follows in this case. This completes the proof of interlacing formula \eqref{eqs:11}. By \eqref{eqsss:43}, the eigenvalues
$\({\la}'_{n}\)_{n\in\N^*}$ of the problem $\(\mathcal{P}_{N}\)$
\eqref{eqs:8} for $M_{N}=0$ satisfy the asymptote  \begin{equation} \label{eqsss:43+}
\frac{{\la'}_n}{n^2\pi^2}=\lim_{\ell_{N}\to L}\frac{\widehat{\la}_n(\ell_N)}{n^2\pi^2}=
{\(\sum_{j=0}^{N-1}\int_{\O_j}\sqrt{\frac{\r_j(x)}{\s_j(x)}}
dx+\int_{\ell_{N-1}}^{L}\sqrt{\frac{\r(x)}{\s(x)}}dx\)^{-2}},\end{equation}
 where
 $(\widehat{\la}_{n})_{n\in\N^*}$ are the eigenvalues of Problem $\(\mathcal{P}_{N-1}\)$ \eqref{eqs:8}, $${\r(x)}=\begin{cases}
                                    {\r_{N-1}(x)}, & x\in \overline{\O}_{N-1}, \\
                                    {\r_N(x)}, & x\in
                                     \overline{\O}_{N},
                                  \end{cases}~~\hbox{ and } ~~{\s(x)}:=\begin{cases}
                                    {\s_{N-1}(x)}, & x\in \overline{\O}_{N-1}, \\
                                    {\s_N(x)}, & x\in
                                     \overline{\O}_{N}.
                                  \end{cases}$$
From \eqref{eqsss:45}-\eqref{eqsss:46}, we have
\begin{equation} \mu_{n-1}^{N,D} \leq\la_{n}\leq\la_{n}'\leq \mu_{n}^{N,D}\leq\la_{n+1}\leq\la'_{n+1}\leq\mu_{n+1}^{N,D},~n\in\N^*,
 \label{eqsss:44}
\end{equation}
and then, by \eqref{eqsss:43+} we get the Weyl's asymptotic formula \eqref{Eqsss:29}. The proof is complete.
\end{proof}
\subsection{Sharp asymptotics of the eigenvalues and spectral gap}\label{Sec3-3}
In this subsection, we establish sharp asymptotic estimates for eigenvalues $(\la_{n})_{n\in\mathbb{N}^{*}}$ of Problem $\(\mathcal{P}_{N}\)$ \eqref{eqs:8}. As consequence, we prove that the spectral gap
"$\big{|}{\la_{n+1}}-{\la_{n}}\big{|}"$ is uniformly
positive. Namely, we enunciate the following result:
\begin{theorem}\label{Weyl} Set $\Lambda^*=\{n+1 : \mu_n^{N,D}=\lambda_{n+1}=\mu_{n+1}^{N,D} \hbox{ for some } n\in\N^*\}$, and let \begin{equation*}\label{q1v5}
Q_{j}^*:=\dfrac{(\ell_{j+1}-\ell_{j})^2}{\gamma_j^2}\ds\int_{\Omega_{j}}
\({\rho_{j}^{-1}}{q_{j}(x)}
-\{{\rho_{j}^{-3}}{\sigma_{j}}(x)\}^{\frac{1}{4}}
\left[\sigma_{j}(x)\left(\{\sigma_{j}\rho_{j}(x)\}^{-\frac{1}{4}}\right)'\right]'\)dx,~ j=0,...,N.
\end{equation*}
Then, the set of eigenvalues $\{\la_n\}_{n\in\N^*}$ of Problem $\(\mathcal{P}_{N}\)$ \eqref{eqs:8} is asymptotically splits into $N+1$ branches $\{\lambda_{n}^j\}_{n\in\N^*}$, $j=0,...,N$, such that:\\
{\bf(a)} for large $n+1\in\Lambda^*$,
\be \sqrt{\lambda_{n+1}^j}=\dfrac{(n+1)\pi}{\o_j^*}
+\dfrac{Q_j^*}{2(n+1)}+\mathcal{O}\(\frac{1}{n^2}\),~j=0,...,N,\label{Manssoura}\ee
{\bf(b)} for large $n+1\in\N^*\backslash\Lambda^*,$ one has the asymptotes:\be \label{eqsss:47} \begin{cases}
\sqrt{\lambda_{n+1}^j}=&\dfrac{n\pi}{\o_j^*}+\dfrac{Q_j^*}{2n} +\dfrac{\gamma\xi_{j}^2(\ell_{j+1})}{M_{j+1}\o_j^*n\pi}+\mathcal{O}\(\frac{1}{n^2}\),~j=0,...,N-1, \\
\sqrt{\lambda_{n+1}^N}=&\dfrac{n\pi}{\o_N^*}+\dfrac{Q_N^*}{2n} +\dfrac{\gamma\xi_{N}^2(\ell_{N})}{M_{N}\o_N^*n\pi}+\mathcal{O}\(\frac{1}{n^2}\),
\end{cases} \ee
where
the quantities
$\xi_j,$ $\gamma$ and $\o_j^*$ are respectively given in \eqref{eqss:21} and  \eqref{eqss:22}.\\
Moreover, \be
{\la_{n+1}}-{\la_{n}}\geq2\gamma\min_{j=0,...,N-1}\left\{\dfrac{\xi_{j}^2(\ell_{j+1})}{M_{j+1}{\o_j^*}^2},\dfrac{\xi_{N}^2
  (\ell_{N})}{M_{N}{\o_N^*}^2}\right\},\hbox{ as }n\to\infty\label{eqsss:30}.
\ee
\end{theorem}
\begin{remark}
 It should be noted that if $n\in\Lambda^*$, then at most $N+1$ of the eigenvalues $\(\widehat\mu_n^{j,D}\)_{n\in\N^*}$ of
the $N+1$ Dirichlet subproblems \eqref{eqs:10} can coincide. This follows from the simplicity of the eigenvalues $\(\widehat\mu_n^{j,D}\)_{n\in\N^*},$ $j = 0,...,N$. For example, in the case of constant coefficients $\rho_j\equiv\sigma_j\equiv1$, $q_j\equiv0$, $\ell_j^*=\ell_{j+1}-\ell_{j}$, and $\ell_j^*=\ell_{j+1}^*$ for all $j\in\{0,...,N\}$. In particular, in this case $\Lambda^*=\N^*$.
Conversely, if ${\frac{\ell_j^*}{\ell_{k}^*}}\in\R_+^*\backslash\mathbb{Q}$ for all $j,k\in\{0,...,N\}$ with $j\not=k$, then
$\Lambda^*\equiv\emptyset.$
\end{remark}
\begin{proof} From the interlacing theorem \ref{interlcing}, we deduce that
 between two consecutive eigenvalues $\mu_{n}$ and $\mu_{n+1}$ there is only one eigenvalue $\la_n$ of Problem $\(\mathcal{P}_{N}\)$ \eqref{eqs:8}. Consequently by  \eqref{eqs:9}, the set of eigenvalues $\{\la_n\}_{n\in\N^*}$ may be decomposed as:
\begin{equation} \{\la_n\}_{n\in\N^*}=\left\{\ds\bigcup_{j=0}^{N}\{\lambda_{n}^j\}_{n\in\N^*}~:~
\sqrt{\la_{n+1}^j}=\sqrt{\widehat\mu_{n}^{j,D}}+\kappa_{n}^j,~j=0,...,N\right\},
\label{eqsss:46v5v4abed}\end{equation}
$\hbox{ for some sequences} \(\kappa_n^j\)_{n\in\N^*}\geq0$. Let $n+1\in\Lambda^*$,
 i.e., $\mu_n^{N,D}=\la_{n+1}=\mu_{n+1}^{N,D}$, then by \eqref{eqsss:46v5v4abed},
 we get   \be\sqrt{\la_{n+1}^j}=\sqrt{\widehat\mu_{n+1}^{j,D}},~j=0,...,N.\label{ja4}\ee Using the modified Liouville transformation (e.g., \cite[Chapter 1]{LEVITAN}),
$$t=\dfrac{\ell_{j+1}-\ell_{j}}{\gamma_j}\int_{\ell_j}^{x}
\sqrt{\frac{\rho_{j}(t)}{\sigma_{j}(t)}}dt +\ell_j ~~\hbox{ and }~~
\hat\phi(t)=(\rho_{j}\sigma_{j}(x))^{\frac{1}{4}}{\phi}(x),~ j=0,...,N.$$
Problem  \eqref{eqs:10} can be written in the
following form
\begin{equation}\label{w'1cvn,}
\left\{
\begin{array}{lll}
-{\hat\phi}''+Q_{j}(t)\hat\phi=\dfrac{\gamma_j^2}{(\ell_{j+1}-\ell_{j})^2}\lambda \hat\phi, ~~t\in\O_j,~ j=0,...,N,\\
\hat\phi(\ell_{j})=\hat\phi(\ell_{j+1})=0,
\end{array}
\right.
\end{equation}
where
\begin{equation*}
Q_{j}:=\dfrac{(\ell_{j+1}-\ell_{j})^2}{\gamma_j^2}\({\rho_{j}^{-1}}{q_{j}}
-\{{\rho_{j}^{-3}}{\sigma_{j}}\}^{\frac{1}{4}}
\left[\sigma_{j}\left(\{\sigma_{j}\rho_{j}\}^{-\frac{1}{4}}\right)'\right]'\),~ j=0,...,N.
\end{equation*}
It is known
(e.g., \cite[Chapter 1]{LEVITAN} and \cite[Chapter
1]{TITCHMARSH}), that the eigenvalues
$\(\widehat\mu_n^{j,D}\)_{n\in\N^*}$ of the $N+1$ Dirichlet subproblems \eqref{w'1cvn,} satisfy the asymptotics \begin{eqnarray}
\label{eqsss:50}
   \sqrt{\widehat\mu_n^{j,D}}=\frac{n\pi}
   {{\o_{j}^*}}+\frac{1}{2n}\ds\int_{\Omega_{j}}Q_{j}(x)d(x)+\mathcal{O}\(\frac{1}{n^2}\),~j=0
   ,...,N.
\end{eqnarray}
Therefore \eqref{Manssoura} is a simple deduction from \eqref{ja4} and \eqref{eqsss:50}. Now, let $n+1\in\N^*\backslash\Lambda^*,$ i.e., $\mu_{n}^{N,D}\not=\mu_{n+1}^{N,D}$. Then by Theorem \ref{interlcing}, \eqref{Manssoura} and  \eqref{eqsss:46v5v4abed}, one has
\begin{equation} \sqrt{\la_{n+1}^j}=\(\dfrac{n\pi}{\o_j^*} +\dfrac{Q_j^*}{2n}+\mathcal{O}\(\frac{1}{n^2}\)\)+\kappa_n^j \hbox{ with } \(\kappa_n^j\)_{n\in\N^*}>0,~j=0,...,N.
\label{eqsss:46v5v4}\end{equation}
 By Proposition \ref{Aymptotics1x}, \eqref{eqss:26} (for $j=0$) and \eqref{eqss:27}, we have
\begin{equation}\begin{cases}\label{eqsss:V4x22}
\dfrac{{\varphi}_j(\ell_{j+1}, \la)}{\s_{j}{\varphi}_j'(\ell_{j+1}, \la)}=\dfrac{\xi_j^2(\ell_{j+1})\sin(\sqrt{\la}\o_{j}^*)
}{\sqrt{\la}\cos(\sqrt{\la}\o_{j}^*)}[1],\\
\dfrac{\varphi_N(\ell_N,\la)}{\s_N\varphi_N'(\ell_N,\la)}=-\dfrac{1}{M_N\lambda}[1],
~j=0,...,N-1,
\end{cases}\end{equation}
where ${\varphi}_j(x, \la)$ are given in \eqref{eqss:15}. Similarly, let $\underline{\psi}_N(x,\la)$ be the solution of the initial
problem \eqref{eqss:3}, \eqref{eqss:8}-\eqref{eqss:10} constructed in
Lemma \ref{Unqunss}. Then by \eqref{eqss:8},
\begin{equation}\begin{cases}\label{eqsspsi:26}
\dfrac{\psi_N(\ell_N,\la)}{\xi_N(\ell_N)\xi_N(\ell_{N+1})}=\dfrac{\sin(\sqrt{\la}\o_N^*)}
{\sqrt{\la}}[1],\\
\dfrac{\xi_N(\ell_{N+1})}{\s_N{\psi}_N'(\ell_N,\la)}=- \dfrac{\xi_N(\ell_N)}{\cos(\sqrt{\la}\o_N^*)}[1],
\end{cases}\end{equation}
and by \eqref{eqss:9}-\eqref{eqss:10}, one has
\begin{equation}\label{eqssV4:23}
\dfrac{{\psi}_j(\ell_{j+1}, \la)}{\s_{j}{\psi}_j'(\ell_{j+1}, \la)}=\dfrac{1}{M_{j+1}\lambda}[1],~j=0,...,N-1,\end{equation}
where ${\psi}_j(x, \la)$ are given by \eqref{eqss:16}. From Theorem \ref{wronskian} and \eqref{eqsss:46v5v4abed}, it follows
\begin{equation}\label{eqsss:V4x24} \begin{cases}
                  \dfrac{\varphi_j(\ell_{j+1}, \la_n^j)}{\s_j(\ell_{j+1})\varphi'_j(\ell_{j+1}, \la_n^j)}=
                  \dfrac{\psi_j(\ell_{j+1}, \la_n^j)}{\s_j(\ell_{j+1})\psi'_j(\ell_{j+1}, \la_n^j)},~j=0,...,N-1,\\
                  \dfrac{\varphi_N(\ell_{N}, \la_n^N)}{\s_N(\ell_{N})\varphi'_N(\ell_{N}, \la_n^N)}=
                  \dfrac{\psi_N(\ell_{N}, \la_n^N)}{\s_N(\ell_{N})\psi'_N(\ell_{N}, \la_n^N)},
                \end{cases}
\end{equation}
where $\varphi_N(x,\la)$ and $\psi_N(x,\la)$ are respectively given by \eqref{eqss:15} and \eqref{eqss:16}. Therefore, from \eqref{eqsss:V4x22}-\eqref{eqssV4:23} and \eqref{eqsss:V4x24}, we get
\begin{equation}\label{eqsss:48R} \begin{cases}
                  \dfrac{\sin(\sqrt{\la_n^j}
\o_{j}^*)}{\cos(\sqrt{\la_n^j}\o_{j}^*)}[1]=\dfrac{\xi_{j}^2(\ell_{j+1})}{M_{j+1}\sqrt{\la_n^j}}[1] ,~j=0,...,N-1,\\
\dfrac{\sin(\sqrt{\la_n^N}\o_{N}^*)}{\cos(\sqrt{\la_n^N}\o_{N}^*)}[1]= \dfrac{\xi_N^2(\ell_N)}{M_N\sqrt{\la_n^N}}[1],
                \end{cases}
\end{equation}
where
the quantities
$\xi_j$, $\gamma$ and $\o_j^*$ are respectively given in \eqref{eqss:21} and  \eqref{eqss:22}. Hence by \eqref{eqsss:46v5v4} and \eqref{eqsss:48R},
\begin{equation}\label{eqsss:48Rbassem}
 \begin{cases}
                  \dfrac{\sin(\kappa_n^j
\o_{j}^*)}{\cos(\kappa_n^j\o_{j}^*)}[1]=\dfrac{\xi_{j}^2(\ell_{j+1})}{M_{j+1}\sqrt{\la_n^j}}[1] ,~j=0,...,N-1,\\
\dfrac{\sin(\kappa_n^N\o_{N}^*)}{\cos(\kappa_n^N\o_{N}^*)}[1]= \dfrac{\xi_N^2(\ell_N)}{M_N\sqrt{\la_n^N}}[1].
                \end{cases}
\end{equation}
It is easy to see that $\kappa_n^j\to0,\hbox{ as }n\to\infty, ~j=0,...,N.$
Therefore, by the Weyl's formula \eqref{Eqsss:29}, \eqref{eqsss:46v5v4} and \eqref{eqsss:48Rbassem}, we get the asymptotes \eqref{eqsss:47}. Now, we  prove the gap condition \eqref{eqsss:30}. By Theorem \ref{simple}, the eigenvalues $(\la_{n})_{n\in\mathbb{N}^{*}}$ of Problem
$\(\mathcal{P}_N\)$ \eqref{eqs:8} are simple, and this implies that, if $n\in\Lambda^*$, then $n+1\in\N^*\backslash\Lambda^*$ (or conversely). Hence, we shall examine the following cases:\\
{\it Case $1$.} If $n\in\Lambda^*,$ then $n+1\in\N^*\backslash\Lambda^*$. Thus by \eqref{Manssoura}-\eqref{eqsss:47}, one has $$
\sqrt{\lambda_{n+1}^j}+\sqrt{\widehat\mu_n^{j,D}}\sim \dfrac{2n\pi}{\o_j^*} ~~\hbox{ and }
\begin{cases}
\sqrt{\lambda_{n+1}^j}-\sqrt{\widehat\mu_n^{j,D}}\sim
\dfrac{\gamma\xi_{j}^2(\ell_{j+1})}{M_{j+1}\o_j^*n\pi},~j=0,...,N-1,\\
\sqrt{\lambda_{n+1}^N}-\sqrt{\widehat\mu_n^{N,D}}
\sim\dfrac{\gamma\xi_{N}^2(\ell_{N})}{M_{N}\o_j^*n\pi},\hbox{ as }n\to\infty.
                \end{cases}
$$
Consequently,  \be\label{eqsss:47V0}
{\lambda_{n+1}^j}-{\widehat\mu_n^{j,D}} \sim\dfrac{2\gamma\xi_{j}^2(\ell_{j+1})}{M_{j+1}{\o_j^{*}}^2}~~ \hbox{ and }~~ {\lambda_{n+1}^N}-{\widehat\mu_n^{N,D}}\sim \dfrac{2\gamma\xi_{N}^2
  (\ell_{N})}{M_{N}{\o_N^*}^2},\hbox{ as }n\to\infty,~j=0,...,N-1.
\ee
{\it Case $2$.} If $n\in\N^*\backslash\Lambda^*,$ then $n+1\in\N^*\backslash\Lambda^*$ or $n+1\in\Lambda^*$. Clearly, if $n+1\in\N^*\backslash\Lambda^*$, then  \eqref{eqsss:47V0} is satisfied. Now, let $n+1\in\Lambda^*$.
By \eqref{Manssoura},  \be\sqrt{\lambda_{n+1}^j}=\sqrt{\widehat\mu_n^{j,D}}=\dfrac{n\pi}{\o_j^*} +\dfrac{Q_j^*}{2n}+\mathcal{O}\(\frac{1}{n^2}\),~j=0,...,N.\label{Maneqsss:47V0}\ee
Since $n\in\N^*\backslash\Lambda^*$, then there exist sequences $\(\kappa_n^j\)_{n\in\N^*}>0$ such that $\sqrt{\la_{n}^j}=\sqrt{\widehat\mu_{n}^{j,D}}-\kappa_{n}^j,~j=0,...,N.$ As above, one has $$ \begin{cases}
\sqrt{\lambda_{n}^j}=&\dfrac{n\pi}{\o_j^*}+\dfrac{Q_j^*}{2n} -\dfrac{\gamma\xi_{j}^2(\ell_{j+1})}{M_{j+1}\o_j^*n\pi}+\mathcal{O}\(\frac{1}{n^2}\),~j=0,...,N-1, \\
\sqrt{\lambda_{n}^N}=&\dfrac{n\pi}{\o_N^*}+\dfrac{Q_N^*}{2n} -\dfrac{\gamma\xi_{N}^2(\ell_{N})}{M_{N}\o_j^*n\pi}+\mathcal{O}\(\frac{1}{n^2}\).
\end{cases} $$
From this and \eqref{Maneqsss:47V0}, it follows
\be\label{kais}
{\lambda_{n+1}^j}-{\lambda_{n}^j} =\dfrac{2\gamma\xi_{j}^2(\ell_{j+1})}{M_{j+1}{\o_j^{*}}^2}~~ \hbox{ and }~~ {\lambda_{n+1}^j}-{\lambda_{n}^j}= \dfrac{2\gamma\xi_{N}^2
  (\ell_{N})}{M_{N}{\o_N^*}^2},\hbox{ as }n\to\infty,~j=0,...,N-1.
\ee
By the interlacing theorem \ref{interlcing}, we have  ${\la_{n+1}^j}-{\la_{n}^j}\geq {\lambda_{n+1}^j}-{\widehat\mu_n^{j,D}}$, $j=0,...,N$. Thus, from \eqref{eqsss:47V0} and \eqref{kais}, one gets \begin{align*}
  \la_{n+1}-\la_{n}& \geq \min_{j=0,...,N}\left\{{\lambda_{n+1}^j}-{\lambda_{n}^j}\right\} \\
  &\geq2\gamma\min_{j=0,...,N-1}\left\{\dfrac{\xi_{j}^2(\ell_{j+1})}{M_{j+1}{\o_j^*}^2},\dfrac{\xi_{N}^2
  (\ell_{N})}{M_{N}{\o_N^*}^2}\right\},\hbox{ as }n\to\infty.
\end{align*}
The proof is complete. \end{proof}
 In the next result, we establish the equivalence between the $\mathcal{H}$-norm of the eigenfunctions
$\({\Phi}_n\)_{n\in\N^*}$ and their first derivative at the right end $x=L$ .
\begin{proposition}\label{Lem1asymptotics210v4equi} Let $\({\Phi}_n\)_{n\in\N^*}$ be
 the sequence of eigenfunctions
 of Problem
$\(\mathcal{P}_N\)$ \eqref{eqs:8} constructed in
Theorem \ref{simple}. One has:
\be  \dfrac{\left\|{\Phi}_n\right\|_{\mathcal{H}}}{\left|\s_{N}(L)\varphi_{N}'(L, \la_n)\right|}
\sim\sqrt{\dfrac{\o_N^*}{{2}}}\dfrac{\gamma\xi_N(L)}
{{n\pi}},\hbox{ as }n\to\infty.\label{eqsssV4kais:31}\ee
where
the quantities
$\xi_N$, $\gamma$ and $\o_N^*$ are respectively given in \eqref{eqss:21} and  \eqref{eqss:22}.
\end{proposition}
\begin{proof} By the change of variables $X=\o_j(x)$, one has
\begin{align*}
&\int_{\Omega_j}\xi_j^2(x)\sin^2(\sqrt{\la}\o_j(x))\rho_j(x)dx=
\int_{\Omega_j}\xi_j^2(x)\cos^2(\sqrt{\la}\o_j(x))\rho_j(x)dx=\dfrac{\o_j^*}{2}[1], \\
&\int_{\Omega_j}\xi_j^2(x)\sin(\sqrt{\la}\o_j(x))\cos(\sqrt{\la}\o_j(x))\rho_j(x)dx=
\frac{\sin^2(\sqrt{\lambda}\o_j^*)}{2\sqrt{\lambda}}[1],~ j=0,...,N,
\end{align*}
where $[1]=1+\mathcal{O}\(\frac{1}{\sqrt{\lambda}}\)$, the quantities
$\xi_j$, $\gamma$ and $\o_j$ are respectively given in \eqref{eqss:21} and  \eqref{eqss:22}. From this
together with \eqref{eqss:26} (for $j=0$) and \eqref{eqss:27},
$$ \left\|\varphi_0\right\|^2_{L^{2}_{\r_0}}=\xi_0^2(\ell_0)
\dfrac{\o_0^*}{2\lambda}[1] \hbox{ and }
\left\|\varphi_1\right\|^2_{L^{2}_{\r_1}}= \(M_{1}\xi_0^*\xi_1(\ell_1)
\sin(\nu\o_0^*)\)^2\dfrac{\o_1^*}{2}[1].$$
Similarly by \eqref{eqss:23}, we get
$$ \left\|\varphi_j\right\|^2_{L^{2}_{\r_j}}=
\(\sqrt{\lambda}^{j-1}\Upsilon_{j-1}\xi_j(\ell_j)\prod_{k=1}^{j}M_{k}\prod_{i=0}^{j-1}
{\sin(\sqrt{\lambda}\o_i^*)}\)^2\dfrac{\o_j^*}{2}[1],~j=2,...,N,$$
where $\Upsilon_j$ are defined by \eqref{eqss:21}.
Thus, by the above asymptotes,
\be \label{ja1}\sum_{j=0}^{N}\left\|\varphi_j\right\|^2_{L^{2}_{\r_j}}=
\left\|\varphi_N\right\|^2_{L^{2}_{\r_N}}[1]=
\(\sqrt{\lambda}^{N-1}\Upsilon_{N-1}\xi_N(\ell_N)\prod_{k=1}^{N}M_{k}\prod_{i=0}^{N-1}
{\sin(\sqrt{\lambda}\o_i^*)}\)^2\dfrac{\o_N^*}{2}[1].\ee
Again by \eqref{eqss:23}-\eqref{eqss:24} and  \eqref{eqss:26}-\eqref{eqss:27}, it follows
\be \sum_{j=1}^{N}M_{j}\varphi_j^2(\ell_j)=\varphi_N^2(\ell_N)[1]=
\(\sqrt{\lambda}^{N-2}\Upsilon_{N-1}\prod_{k=1}^{N}M_{k}\prod_{i=0}^{N-1}
{\sin(\sqrt{\lambda}\o_i^*)}\)^2[1],\ee
and \be \label{ja2}{(-1)^{N}\s_{N}\varphi_N'(L)}= \(\sqrt{\la}^N\Upsilon_{N-1}\xi_N(\ell_N)\prod_{k=1}^{N}M_{k}\prod_{k=0}^{j-1}
{\sin(\sqrt{\lambda}\o_k^*)}\)\dfrac{\cos(\sqrt{\lambda}\o_N^*)}{\xi_N(L)}[1].\ee
Thus, by \eqref{eqss:30}, \eqref{ja1}-\eqref{ja2}, one gets
$$ \dfrac{\left\|{\Phi}_n\right\|_{\mathcal{H}}^2}{\left|\s_{N}(L)\varphi_{N}'(L, \la_n)\right|^2}
=\dfrac{\left\|\varphi_N(x, \la_n)\right\|_{L^{2}_{\r_N}}}
{\left|{\s_{N}(L)\varphi_{N}'}(L, \la_n)\right|^2}[1]=\dfrac{{\o_N^*\xi_N^2(L)}}
{{2}{\lambda_n}\cos^2(\sqrt{\lambda_n}\o_N^*)}[1].$$
Or equivalantly (by Theorem \ref{Weyl}), \be \label{ja3} \dfrac{\left\|{\Phi}_n\right\|_{\mathcal{H}}^2}{\left|\s_{N}(L)\varphi_{N}'(L, \la_n^j)\right|^2}
=\dfrac{{\o_N^*\xi_N^2(L)}}
{{2}{\lambda_n^j}\(1-\sin^2(\sqrt{\lambda_n^j}\o_N^*)\)}[1],~j=0
   ,...,N.\ee
It is easy to see from \eqref{Manssoura}-\eqref{eqsss:47} that
$$\left|\sin\(\sqrt{\lambda_n^j}\o_N^*\)\right|\leq\dfrac{C_j}{n},~j=0
   ,...,N,$$
for some constants $C_j>0$. Therefore, from this and \eqref{ja3} together with the Weyl's formula \eqref{Eqsss:29}, we get
the equivalence \eqref{eqsssV4kais:31}. The proof is complete.
\end{proof}
\section{Controllability}\label{Sec4}
\setcounter{equation}{0}
In this section, we prove our main results, namely the null controllability of System
\eqref{eqs:1}-\eqref{eqs:3}, and then, the exact controllability of the Schr\"{o}dinger model \eqref{Scontrolanis}.
\subsection{Null controllability of the heat model \eqref{eqs:1}-\eqref{eqs:3}}
In this subsection, we prove Theorem \ref{MAIN}. We do it by reducing the control problem to
problem of moments. Then, we will solve this problem of moments
using the theory developed in \cite{RussellF, RussellF1}. To this end, let us consider the
so-called adjoint problem, that is,
\begin{equation}
\label{eqref4s4}
\left\{
\begin{array}{ll}
\(\r_{j}(x)\partial_t \hat u_j+\partial_x\(\s_{j}(x)\partial_x \hat u_j\)-q_{j}(x)\hat u_j\)(t,x)=0,& t>0,~ x\in {\O}_j,~j = 0,...,N, \\
\hat u_{j-1}(t, \ell_j)=\hat z_j(t)=\hat u_{j}(t, \ell_j),&t>0,~j=1,...,N, \\
\(\s_{j-1}(\ell_j)\partial_x \hat u_{j-1}-\s_{j}(\ell_j)\partial_x \hat u_j\)(t,\ell_j)= M_j\partial_t \hat z_{j}(t),&t>0,~j=1,...,N,  \\
\hat u_{0}(t,0)=0,~\hat u_{N}(t,L)=0 & t>0,
\end{array}
\right.
\end{equation} with final data at $t = T>0$ given by
\begin{eqnarray}
\label{eqref4s5}
\left\{
\begin{array}{ll}
  \hat u_j(T,x)=\hat u^{T}_j,~x\in\O_j,&j=0,...,N, \\
  \hat z_j(T)=\hat z^{T}_j,&j = 1, ...,N.
\end{array}
 \right.
\end{eqnarray}
By letting $\widehat{U}= \(\(\hat u_j\)_{j=0}^{N},\(\hat z_j\)_{j=1}^{N}\)^{\top}$, the above problem can be written
as \begin{equation*}
\partial_t\widehat{U}(t)=\mathcal{A}\widehat{U}(t),~~
\widehat{U}(T)=\widehat{U}^{T},~~t\in\(0,\infty\),
\end{equation*}
where $\mathcal{A}$ is defined in \eqref{eqss:1} and
${\widehat{U}}^T=\(\(\hat u^{T}_j\)_{j=0}^{N},\(\hat z^{T}_j\)_{j=1}^{N}\)^{\top}$.
Then, we have the following characterization of the null-controllability property.
\begin{lemma}
System \eqref{eqs:1}-\eqref{eqs:3} is null-controllable in time $T > 0$, if and only if,
for any initial data ${U}^0=\(\(u^{0}_j\)_{j=0}^{N},\(z^{0}_j\)_{j=1}^{N}\)^{\top}\in
\mathcal{H}$, there exists a control function
$h(t)\in H^{1}(0, T ),$ such that, for any
${\widehat{U}}^T=\(\(\hat u^{T}_j\)_{j=0}^{N},\(\hat z^{T}_j\)_{j=1}^{N}\)^{\top}\in
\mathcal{H}$
\begin{equation} \left\langle U^0,\(\(\hat u_j(0,x)\)_{j=0}^{N},\(\hat z_j(0)\)_{j=1}^{N}\)^{\top}\right\rangle_{\mathcal{H}}=
\s_{N}\(L\)\int_{0}^{T}h(t)\partial_x\hat{u}_N(t,L)dt\label{eqref4s6}\end{equation}
where $\widehat{U}= \(\(\hat u_j\)_{j=0}^{N},\(\hat z_j\)_{j=1}^{N}\)^{\top}$
is the solution of the adjoint problem \eqref{eqref4s4}-\eqref{eqref4s5}.
\end{lemma}
\begin{proof} We proceed as in the classical duality approach. We first multiply the
$N+1$ equations in \eqref{eqs:1} by
$\(\hat u_j\)_{j=0}^{N}$, to obtain
\begin{equation*} \sum_{j=0}^{N}\int_{\ell_j}^{\ell_{j+1}}\int_{0}^{T}\partial_t u_j\hat u_jdt \r_{j}(x)dx
=\int_{0}^{T}\sum_{j=0}^{N}\int_{\ell_j}^{\ell_{j+1}}\big(\partial_x
\(\s_{j}(x)\partial_x u_j\)-q_{j}(x)u_j\big)\hat u_jdt dx,\end{equation*}
where $\widehat{U}= \(\(\hat u_j\)_{j=0}^{N},\(\hat z_j\)_{j=1}^{N}\)^{\top}$
is the solution of Problem \eqref{eqref4s4}-\eqref{eqref4s5}.
Integration by parts leads to \begin{alignat}{3}\label{eqref4s7}
\sum_{j=0}^{N}\int_{\ell_j}^{\ell_{j+1}}u_j\hat u_j\big|_{t=0}^{t=T}\r_{j}(x)dx=
\int_{0}^{T}\sum_{j=0}^{N} \(\s_{j}(x)\partial_x u_j\hat u_j\big|_{x=\ell_j}^{x=\ell_{j+1}}-
\s_{j}(x)\partial_x \hat u_j u_j\big|_{x=\ell_j}^{x=\ell_{j+1}}\)dt.
                              \end{alignat}
Since $\displaystyle\sum_{j=0}^{N}\s_{j}(x)\partial_x u_j\hat
u_j\big|_{x=\ell_j}^{x=\ell_{j+1}}=-\sum_{j=1}^{N}M_j\partial_t z_{j}\hat z_{j}(t)$, and $$
\sum_{j=0}^{N}\s_{j}(x)\partial_x \hat u_j
u_j\big|_{x=\ell_j}^{x=\ell_{j+1}}= \sum_{j=1}^{N}M_j\partial_t \hat z_{j}z_{j}(t)+
\s_{N}(L)h(t)\partial_x \hat u_j(t,L),$$
then by \eqref{eqref4s7}, one gets $$
\sum_{j=0}^{N}\int_{\ell_j}^{\ell_{j+1}}u_j\hat u_j\big|_{t=0}^{t=T}\r_{j}(x)dx +
\sum_{j=1}^{N}M_jz_j\hat z_j(t)\big|_{t=0}^{t=T}=-\int_{0}^{T}\s_{N}(L)
\partial_x \hat u_Nu_N(t,L)dt.$$
Equivalently,  \begin{equation} \label{eqref4s8}
\left\langle U(T) , \widehat{U}^T\right\rangle_{\mathcal{H}}=\left\langle U^0,\widehat{U}(0)\right\rangle_{\mathcal{H}}-\int_{0}^{T}\s_{N}(L)
\partial_x \hat u_Nu_N(t,L)dt,\end{equation}
where $$U(T):=\(\(\ u_j(T,x)\)_{j=0}^{N},\(z_j(T)\)_{j=1}^{N}\)^{\top}\hbox{ and }
\widehat{U}(0):=\(\(\hat u_j(0,x)\)_{j=0}^{N},\(\hat z_j(0)\)_{j=1}^{N}\)^{\top}.$$
 Now, we assume that \eqref{eqref4s6}
 holds. Then by \eqref{eqref4s8}, one has \begin{eqnarray} \label{eqref4s9} \left\{
\begin{array}{ll}
  u_j(T,x)=0,&\forall x\in\O_j,~\forall
j\in\{0,...,N\}, \\
  z_j(T)=0,&\forall j\in\{1,...,N\}.
\end{array}
 \right.
\end{eqnarray} Thus, the solution $U$ is controllable to zero and $h(t)$ is
a control of Problem \eqref{eqs:1}-\eqref{eqs:3}. Conversely, if $h(t)$ is a control of
Problem \eqref{eqs:1}-\eqref{eqs:3} for which \eqref{eqref4s9} holds. Thus by \eqref{eqref4s8},
we get \eqref{eqref4s6}. The Lemma is proved.
\end{proof}

We are now ready to reduce the control problem \eqref{eqs:1}-\eqref{eqs:3}
to a moment problem. Let $\({\Phi}_n\)_{n\in\N^*}$ be
 the sequence of eigenfunctions
 of Problem
$\(\mathcal{P}_N\)$ \eqref{eqs:8} constructed in
Theorem \ref{simple}, then any  terminal data
 $\widehat U^T:=\(\(\hat u^{T}_j\)_{j=0}^N,\(\hat z^{T}_j\)_{j=1}^{N}
\)^{\top}\in\mathcal{H}$ for the adjoint problem \eqref{eqref4s4}-\eqref{eqref4s5} can be written as \begin{eqnarray*}
\widehat U^{T}&=&\sum_{n\in\mathbb{N}^{*}}\dfrac{\langle \widehat U^{T}_n, {\Phi}_n
\rangle_{\mathcal{H}}}{\|{\Phi}_n\|^{2}}{{\Phi}}_{n}(x),\end{eqnarray*}
where the Fourier coefficients $\widehat U_{n}^{T}=
\dfrac{\langle \widehat U^{T}_n, \Phi_{n}
\rangle_{\mathcal{H}}}{\|\Phi_{n}\|^{2}}$, $n\in\mathbb{N}^{*}$,
belong to $\ell^2(\N)$. Hence, the solution $\widehat{U}(t,x)=\(\(\hat u_{j}(t,x)\)_{j=0}^{N},
\(\hat z_{j}(t,x)\)_{j=1}^{N}\)^{\top}$
of \eqref{eqref4s4}-\eqref{eqref4s5} is given
by   $$\widehat{U}(t,x)=\sum_{n\in\mathbb{N}^{*}}\widehat U_{n}^{T}
 e^{-\la_{n}(T-t)}\Phi_n(x),$$
and we have \begin{eqnarray*} \partial_x\widehat{U}(t,L)=\partial_x\hat u_{N}(t,L)=
\sum_{n\in\mathbb{N}^{*}} \widehat U_{n}^{T} e^{-\la_{n}(T-t)}\varphi_{N}'(L, \la_n),\end{eqnarray*}
where
${\varphi}_N(x, \la)$ is defined by \eqref{eqss:15}. Using this fact in
\eqref{eqref4s6}, on gets the following lemma.
\begin{lemma}Problem \eqref{eqs:1}-\eqref{eqs:3} is null-controllable
in time $T > 0$ if and
only if for any \begin{eqnarray*}
U^{0}&=&\sum_{n\in\mathbb{N}^{*}}\dfrac{\langle U^{0}, \Phi_{n}
\rangle_{\mathcal{H}}}{\|\Phi_{n}\|^{2}}{\Phi}_{n}\in \mathcal{H},\end{eqnarray*} there
exists a function $h(t)\in H^{1}(0,T)$ such that \begin{equation} e^{-\la_{n}T}\langle U^{0}, \Phi_{n}
\rangle_{\mathcal{H}}= \s_{N}(L)\varphi_{N}'(L, \la_n)
\int_{0}^{T}h(T-t)e^{-\la_{n}t}dt,~\forall n\in\mathbb{N}^{*}.\label{eqref4s10}\end{equation}
\end{lemma}
We are now in a position to prove Theorem \ref{MAIN}.
\begin{proof}
From the Weyl's formula \eqref{Eqsss:29}, $$\sum_{n\in
\mathbb{N}^{*}}\frac{1}{\la_{n}}<\infty,$$ and then by Theorem \ref{Weyl}, we deduce that there exists a
biorthogonal sequence $(\Theta_{n}(t))_{n\in \mathbb{N}^{*}}$ to the
family of exponential functions $(e^{-\la_{n}t})_{n\in
\mathbb{N}^{*}}$ (see \cite{RussellF, RussellF1}) such that
$$\int_{0}^{T}\Theta_{n}(t)e^{{-\la_{m}t}}dt=\delta_{nm}=\left\{
\begin{array}{l}
                                                      1,~~if~~n=m, \\
                                                      0,~~if~~n \neq m.
                                                    \end{array}
 \right.
$$
Again by \eqref{Eqsss:29} together with the
general theory developed in \cite{RussellF}, it follows that
there exists constants $C_j(T)>0~(\hbox{depending on}~T)$ and $\widehat C_j>0$
such that for any $j\in\N$,
\begin{equation} \left\|\Theta_{n}(t)\right\|_{H^{j}(0,T)}\leq C_j(T)
e^{\widehat C_j n},~n \in \mathbb{N}^{*}.\label{eqref4s11}\end{equation}
Let $\({\Phi}_n\)_{n\in\N^*}$ be the sequence of eigenfunctions
 of Problem
$\(\mathcal{P}_N\)$ \eqref{eqs:8} constructed in
Theorem \ref{simple}, then
\begin{equation*} {\Phi}_n(L)= \varphi_{N}'(L, \la_n)\not=0,~~\forall n\in\N^*,\end{equation*}
where
${\varphi}_N(x, \la)$ is defined by \eqref{eqss:15}. Indeed,
 if ${\Phi}_n'(L)=0$, then the restriction
${\varphi}_N(x, \la)$ of ${\Phi}_n$ to $\O_N$ satisfies,
${\varphi}_N(L, \la_n)= {\varphi}_N'(L, \la_n)=0$. Thus,
${\varphi}_N(x, \la_n)=0$, a contradiction. Therefore from the above,
we infer that an explicit formal solution of
the moment problem  \eqref{eqref4s10} is given by
\begin{equation} h(T-t)=k(t)= \sum_{n\in
\mathbb{N}^{*}}\dfrac{\langle U^{0}, \Phi_{n}
\rangle_{\mathcal{H}}}{\s_{N}{\varphi}_N'(L, \la_n)}e^{-\la_{n}T}\Theta_{n}(t).\label{eqref4s12}\end{equation}
As a consequence, the task consists in showing that
the series $k(t)$ convergence in $H^1(0,T)$.
 From \eqref{eqref4s11} and \eqref{eqref4s12}, we obtain by Cauchy-Schwarz inequality that
\begin{align}
  \left\|h \right\|_{H^1(0,T)} & \leq C_1(T)\(\sum_{n\in
\mathbb{N}^{*}}\dfrac{{\left\|{\Phi}_n\right\|_{\mathcal{H}}}}
{\left|\s_{N}(L)\varphi_{N}'(L,\la_n)\right|}
e^{-\la_{n}T+\widehat C_1n}\)\left\|U^{0}\right\|_{\mathcal{H}}.
  \label{abedrahman}
\end{align}
Therefore from Proposition \ref{Lem1asymptotics210v4equi}, the Weyl's formula \eqref{Eqsss:29}
and \eqref{abedrahman}, it follows
\begin{align*}
  \nonumber\left\|h \right\|_{H^1(0,T)} & \leq C(T)\sqrt{\dfrac{\o_N^*}{{2}}}\xi_N(L)\sum_{n\in
\mathbb{N}^{*}} \dfrac{\gamma e^{-\frac{n\pi T}{\gamma}}}{n\pi}
\left\|U^{0}\right\|_{\mathcal{H}},\end{align*}
for some new constant $C(T)>0,$ where the quantities
$\xi_j,$  $\gamma$ and $\o_j^*$ are respectively given in \eqref{eqss:21} and  \eqref{eqss:22}. This proves the convergence of the series $h(t)$ and finishes the proof
of Theorem \ref{MAIN}.
\end{proof}
\subsection{Exact controllability of the Schr\"{o}dinger model \eqref{Scontrolanis}}\label{sec4-2}
In this subsection, we prove Theorem \ref{hjAnis}.
\begin{proof}
 By means of Lions HUM method (see \cite{Lions}), controllability properties of the Schr\"{o}dinger model \eqref{Scontrolanis} can be reduced to suitable observability inequalities for the adjoint system. As \eqref{Scontrolanis} is reversible in time, we are reduced to the same
system without control. Let $\widehat{U}:=\big(\hat u_0(t,x),\hat u_1(t,x),\hat z(t)\big)^{\top}$ be the unique solution of Problem  \eqref{Scontrolanis} with $h(t)\equiv0$. It is easy to show that   \be\widehat{U}(t,x):=\sum\limits_{n\in \N^*}
c_{n}e^{i{\la}_n t} \widehat\Phi_n\(x\)\in C\([0,T],~ H^{1}_0(0,1) \times \mathbb{C}\),~c_{n}\in\ell^2(\N^*),\label{ali}\ee
where $(\lambda_{n})_{n\N^*}$ are the eigenvalues of the spectral problem \begin{align}\label{eqs:8dfavdonin}
                       &-\p_j''=\la\p_j,~x\in \(\ell_j,\ell_{j+1}\),~j=0,1,  \\
                       &\p_{0}(\ell_{1})=\p_{1}(\ell_{1}),
                       ~\(\p_{0}'-\p_{1}'\)(\ell_{1})=\la \p_{0}(\ell_{1}), \label{eqs:8dfavdonin0}\\
                       &\p_0(\ell_0)=\p_0(0)=0,~ \p_{1}(\ell_{2})=\p_1(1)=0, \label{eqs:8dfavdonin1}
                        \end{align}
and $\(\widehat{\Phi}_n\)_{n\in\N^*}$ are the associated eigenfunctions, which are normalized in the Hilbert space $\mathcal H=\ds\prod_{j=0}^{1} L^{2}(\ell_j,\ell_{j+1})\times \C$ so that $\ds\lim_{n\to\infty}\|\widehat{\Phi}_n\|_{\mathcal{H}}=1$. Consequently, the task now is to
prove following observability inequality: \begin{equation}\label{obghvbds}
\int_{0}^{T} \left|\partial_x{\hat u_0}(t,0) \right|^{2}dt \asymp
\|\widehat{U}(0,x)\|_{H^1_0(0,1) \times \mathbb{C}}^{2},~\forall~T>0.
\end{equation}
To this end, following Lemma \ref{Unqunss} and Proposition \ref{Aymptotics1x}, it easy to see that the problem determined by Equations \eqref{eqs:8dfavdonin}-\eqref{eqs:8dfavdonin0}, and the initial conditions $\p_0(0)= \p_0'(0)-1=0$, has a unique solution \be \label{eqss:15Avdohan0}
 \underline{\varphi}(x,\lambda):=\left\{
\begin{array}{ll}
{\varphi}_0(x,\la) =\dfrac{\sin(\sqrt{\lambda} x)}{\sqrt{\lambda}},~x\in\[0,\ell_{1}\], \\
{\varphi}_1(x,\la) =\dfrac{\sin(\sqrt{\lambda} x)}{\sqrt{\lambda}}-\sin(\sqrt{\lambda} \ell_1)\sin(\sqrt{\lambda}(x-\ell_1)),~~x\in\[\ell_1,1\].
\end{array}\right.\ee
Similarly, the problem determined by Equations \eqref{eqs:8dfavdonin}-\eqref{eqs:8dfavdonin0}, and the initial conditions $\p_0(1)= \p_0'(1)+1=0$, has a unique solution $\underline{\psi}(x,\la)$ :\be \label{eqss:15Avdohan1}
 \underline{\psi}:=\left\{
\begin{array}{ll}
{\psi}_0(x,\la) =\dfrac{\sin(\sqrt{\lambda}(1- x))}{\sqrt{\lambda}}-\sin(\sqrt{\lambda}(1-\ell_1))\sin(\sqrt{\lambda}(\ell_1-x)) ,~x\in\[0,\ell_{1}\], \\
{\psi}_1(x,\la) = \dfrac{\sin(\sqrt{\lambda}(1- x))}{\sqrt{\lambda}},~~x\in\[\ell_1,1\].
\end{array}\right.\ee
Following an argument
similar to that in the proof of Theorem \ref{simple}, we deduce that the eigenfunctions $\({\Phi}_n(x)\)_{n\in\N^*}$ associated with Problem \eqref{eqs:8dfavdonin}-\eqref{eqs:8dfavdonin1} taken the form \begin{equation*}
\({\Phi}_n(x)\)_{n\in\N^*}:=\(\big({\varphi}_j(x,\la_n)\big)_{j=0}^1,
{\varphi}_1(\ell_1,\la_n)\)_{n\in\N^*}^\top,
~x\in\[\ell_j,\ell_{1+1}\],~j=0,1,\end{equation*}
where ${\varphi}_j(x,\la)$ are given by \eqref{eqss:15Avdohan0}. Consequently, the eigenfunctions \be \label{hedieqss:15Avdohan0}
\widehat{\Phi}_n(x)
:=\dfrac{{\Phi}_n(x)}{\|{\Phi}_n\|_{\mathcal{H}}},~\forall n\in\N^*,\ee
can be chosen to constitute an  orthonormal basis of $\mathcal H$, and then, the space $H^1_0(0,1) \times \mathbb{C}$ can be characterized as
\be H^1_0(0,1) \times \mathbb{C}=
\left\{u(x)=\sum\limits_{n\in\N^*}c_n\widehat\Phi_n(x)~:
 ~\|u\|_{H^1_0(0,1) \times \mathbb{C}}^2=\sum\limits_{n\in\N^*}{\la_n}|c_n|^{2}<\infty\right\}.\label{hedieqss:15Avdohan0gogou}\ee
By \eqref{eqss:15Avdohan0} and \eqref{hedieqss:15Avdohan0}, a simple calculation yields \be \label{castrohedieqss:15Avdohan0}|\widehat{\Phi}_n'(0)|
:=\dfrac{|{\Phi}_n'(0)|}{\|{\Phi}_n\|_{\mathcal{H}}}= \sqrt{\dfrac{2}{1-\ell_1}}\dfrac{1}{\left|\sin\(\sqrt{\lambda_n}\ell_1\)\right|}\[1\],\ee
where $\[1\]=1+\mathcal{O}\(\frac{1}{\sqrt{\lambda_n}}\).$
From Theorem \ref{wronskian}, $\dfrac{\varphi_{j}(\ell_1,\la)}{\varphi'_{j}(\ell_1,\la)}=
\dfrac{\psi_{j}(\ell_1,\la)}{\psi'_{j}(\ell_1,\la)},~j=0,1,$ and then by \eqref{eqss:15Avdohan0} and \eqref{eqss:15Avdohan1}, one has
\be \frac{\sin( \sqrt{\lambda_n}\ell_1)}{\cos( \sqrt{\lambda_n}\ell_1)}=\frac{1}{\sqrt{\lambda_n}}[1] ~~\hbox{ and }~~\frac{\sin(\sqrt{\lambda_n}(1-\ell_1))}{\cos(\sqrt{\lambda_n}(1-\ell_1))}=
\frac{1}{\sqrt{\lambda_n}}[1].\label{abdguo}\ee
Let $\left\{\mu_{n}\right\}_{1}^{\infty}=\ds\left
\{\(\frac{n\pi}{{\ell_{1}}}\)^2\right\}_{1}^{\infty}\bigcup\left\{\(\frac{n\pi}{{1-\ell_{1}}}\)^2\right\}_{1}^{\infty}$. Then, under Condition \eqref{brous} together with Theorem \ref{interlcing},
 \be
 0<\la_1<\inf\left\{\(\frac{\pi}{{\ell_{1}}}\)^2, \(\frac{\pi}{{1-\ell_{1}}}\)^2\right\}, ~~ \mu_{n}<\la_{n+1}<\mu_{n+1},~\forall n\in\N^*.
 \label{g35allahguo}
\ee and \be \lambda_n\sim{n^2\pi^2},\label{waveguo}\ee
Following an argument
similar to that in the proof of Theorem \ref{simple}, using \eqref{abdguo}-\eqref{g35allahguo} and \eqref{waveguo}, we deduce that the set of eigenvalues $\{\la_n\}_{n\in\N^*}$ is asymptotically splits into two branches $\{\lambda_{n}^j\}_{n\in\N^*}$, $j=0,1$, such that: \be \sqrt{\lambda_{n}^0}=\frac{n\pi}{\ell_{1}}+ \frac{1}{\ell_{1}n\pi}+\mathcal{O}\(\frac{1}{n^2}\)~\hbox{ and }~\sqrt{\lambda_{n}^1}=\frac{n\pi}{1-\ell_{1}}+ \frac{1}{(1-\ell_{1})n\pi}+\mathcal{O}\(\frac{1}{n^2}\).\label{waveguohedi}\ee
Consequently, \be
{\la_{n+1}}-{\la_{n}}\geq2\min\left\{\dfrac{1}{{\ell_1}^2},\dfrac{1}{{(1-\ell_1)}^2}\right\},\hbox{ as }n\to\infty\label{edwardeqsss:30},
\ee
and since $1=1-\ell_1+\ell_1$, by  \eqref{castrohedieqss:15Avdohan0}, one gets the equivalence \be \label{castro1hedieqss:15Avdohan0}|\widehat{\Phi}_n'(0)|\asymp {n}\asymp
\frac{1}{\left|\sin\(\sqrt{\lambda_n^j}\ell_1\)\right|}
,~j=0,1.\ee
From  \eqref{g35allahguo}-\eqref{waveguo}, we find that the Beurling upper density of the eigenvalues $(\lambda_{n})_{n\in\N^*}$,
$$ D^+\(\la_n\):= \ds \lim_{r\to\infty}
\frac{n^+\(r,\lambda_n\)}{r}=\lim_{n\to\infty}\frac{1}{n\pi}=0,$$
where $n^+\(r ,\lambda_n\)$ denotes the maximum number of terms of the sequence $(\lambda_{n})_{n\in\N^*}$ contained
in an interval of length $r$. Therefore Beurling's Theorem  (e.g., \cite{DagerZuazua}) states that for any $T>0$, the family $\(e^{i\lambda n t}\)_{n\in\N^*}$ forms a Riesz basis in $L^2(0, T)$. Furthermore, for every $T>0,$
\begin{equation*}
\int_{0}^{T}\left|\sum\limits_{n\in\mathbb{\N^*}}\chi_n e^{i{\la}_nt}\right|^{2}dt \asymp\sum\limits_{n\in
\N^*}\left|\chi_{n}\right|^{2},
\end{equation*}
for all sequences of complex numbers $(\chi_n)_{n\in\N^*}$. Let $\chi_{n}=c_{n}\widehat\Phi_n\(0\)$, then by  \eqref{ali} and \eqref{obghvbds}, $$
\int_{0}^{T}|\partial_x{\hat u_0}(t,0)|^{2}dt \asymp\sum\limits_{n\in
\N^*}\left|c_{n}\widehat\Phi_n\(0\)\right|^{2},~\forall ~T>0.
$$ Therfore, from this, \eqref{hedieqss:15Avdohan0gogou} and \eqref{castro1hedieqss:15Avdohan0}, we get the observability inequality \eqref{obghvbds}. The proof is complete.
 \end{proof}
\subsection*{Acknowledgments} The second author wishes to thank Enrique Zuazua for many fruitful discussions on the subject and for suggestions concerning the problem.

\end{document}